\documentclass[11pt]{amsart}
\usepackage{amsmath,amsthm,amssymb,mathtools,yhmath}
\usepackage{graphicx,adjustbox}
\usepackage{stmaryrd}
\usepackage{amsfonts}
\usepackage{enumerate, url}
\usepackage[margin=1in]{geometry}
\usepackage{caption}
\usepackage{algorithm}
\usepackage{algpseudocode}
\usepackage{mleftright}
\usepackage{hyperref}
\usepackage{tikz}
\usepackage{tabu}
\usepackage{float}
\usepackage{indentfirst}
\usepackage{mathdots} %\ddots, \vdots, \iddots
\usepackage[all]{xy}
\usepackage{subfigure}
\usepackage{tikz-cd}  % graph
\usetikzlibrary{positioning}
\usepackage{xcolor}
\hypersetup{
    colorlinks,
    linkcolor={red!50!black},
    citecolor={blue!50!black},
    urlcolor={blue!80!black}
}
\usepackage{hyperref}
\usepackage{tabu}
\usepackage{adjustbox}

\newtheorem{theorem}{Theorem}[section]
\newtheorem{proposition}[theorem]{Proposition}
\newtheorem{proposition/definition}[theorem]{Proposition/Definition}
\newtheorem{lemma}[theorem]{Lemma}
\newtheorem{corollary}[theorem]{Corollary}

\theoremstyle{definition}

\newtheorem{example}[theorem]{Example}

\newtheorem{problem}[theorem]{Problem}

\theoremstyle{remark}

\newtheorem{remark}[theorem]{Remark}

\newcommand{\Rmnum}[1]{\expandafter\@slowromancap\Romannumeral #1@} 
\newcommand{\tp}{{\scriptscriptstyle\mathsf{T}}}

\let\O\undefined

\DeclareMathOperator{\O}{O}

\DeclareMathOperator{\Vol}{Vol}

\DeclareMathOperator{\hess}{Hess}

\DeclareMathOperator{\End}{End}

\DeclareMathOperator{\rank}{rank}
\DeclareMathOperator{\diag}{diag}

\DeclareMathOperator{\sign}{sign}

\DeclareMathOperator{\grad}{grad}
\DeclareMathOperator{\Hess}{Hess}
\DeclareMathOperator{\spa}{span}
\DeclareMathOperator{\Det}{Det}
\begin{document}
\title{The sparseness of g-convex functions}
\author[Y.~Wang]{Yu Wang}
\address{KLMM, Academy of Mathematics and Systems Science, Chinese Academy of Sciences, Beijing 100190, China}
\email{wangyu2020@amss.ac.cn}
\author[K.~Ye]{Ke Ye }
\address{KLMM, Academy of Mathematics and Systems Science, Chinese Academy of Sciences, Beijing 100190, China}
\email{keyk@amss.ac.cn}
\date{\today}
%%%%%%%%%%%%%%%%%%%%%%%%%%%%%%%%%%%%%%%%%%%%%%%%%%%%%%%%%%%%%%%%%%
\begin{abstract}
The g-convexity of functions on manifolds is a generalization of the convexity of functions on $\mathbb{R}^n$.  It plays an essential role in both differential geometry and non-convex optimization theory.  This paper is concerned with g-convex smooth functions on manifolds.   We establish criteria for the existence of a Riemannian metric (or connection) with respect to which a given function is g-convex.  Using these criteria,  we obtain three sparseness results for g-convex functions: \emph{(1)} The set of g-convex functions on a compact manifold is nowhere dense in the space of smooth functions.  \emph{(2)} Most polynomials on $\mathbb{R}^n$ that is g-convex with respect to some geodesically complete connection has at most one critical point.  \emph{(3)} The density of g-convex univariate (resp.  quadratic,  monomial,  additively separable) polynomials asymptotically decreases to zero.  
\end{abstract}
\maketitle
%%%%%%%%%%%%%%%%%%%%%%%%%%%%%%%%%%%%%%%%%%%%%%%%%%%%%%%%%%%%%%%%%%
%%%%%%%%%%%%%%%%%%%%%%%%%%%%%%%%%%%%%%%%%%%%%%%%%%%%%%%%%%%%%%%%%%
\section{Introduction}
Convex functions on Euclidean spaces lie at the core of various branches of mathematics.  The study of such functions ranges over optimization theory \cite{rockafellar-1970a,boyd2004,NY2018},  convex geometry \cite{MR1242973,HDW},  functional analysis \cite{Fenchel49,  KZ05},  combinatorics \cite{Kalai95, Grunbaum03} and algebraic geometry \cite{CLS11,MS15}.  On a manifold,  convex functions are generalized to g-convex functions and they play an increasingly important role in both pure and applied mathematics.  For instance,  the existence of a non-constant g-convex function on a Riemannian manifold governs its global geometry and topology \cite{BO69,CG72,Yau74}.  The g-convexity of the norm function on a Lie group is essential in geometric invariant theory and algebraic complexity theory  \cite{KN79,BFGOWW19}.  In the context of Riemannian optimization \cite{EAS99,Ab2007,YWL22,BN2023},  g-convex functions were first studied 30 years ago \cite{Rapcsak91,UC1994}.  However,  only recently has the community become aware of their great importance \cite{Wiesel12,Bavcak14,SH15,SVY18}.  The key observation is: \emph{functions that are not convex can be g-convex}.  Examples include:
\begin{itemize}
\item[$\diamond$] Rosenbrock banana \cite[Section~3.6]{UC1994}: $f(x)= a(x_2-x_1^2)^2+(b-x_1)^2$ defined on $\mathbb{R}^2$,  where $a,b$ are positive real numbers.
\item[$\diamond$] Reformulation of Brascsamp-Lieb function \cite{SH15,SVY18}: $f(X) = \sum_{i=1}^n a_i \log \det (A_j X A_j^\tp) - \log \det (X)$ defined on the cone $\mathsf{S}^2_{++}(\mathbb{R}^q)$ consisting of $q \times q$ positive definite matrices,  where $A_1,\dots,  A_n$ are $p \times q$ matrices and $a_1,\dots,  a_n$ are positive real numbers.
\item[$\diamond$] Logarithm of a positive polynomial \cite{SV17}: $f(x) = \log p(x_1,\dots,  x_n) $ defined on $(0,\infty)^n$,  where $p$ is a polynomial with non-negative coefficients.
\item[$\diamond$] Karcher mean \cite{SH15}:  $f(X) = \sum_{i=1}^p \lVert \log (X^{-1/2} A_i X^{-1/2}) \rVert_F^2$ defined on $\mathsf{S}^2_{++}(\mathbb{R}^q)$,  where $A_1$,$\dots$,  $A_p$ are $q\times q$ positive definite matrices.
\item[$\diamond$] The orbit norm \cite{KN79,BFGOWW19}: $f(t) =  \lVert t \cdot v \rVert^2$ on  torus $T$,  where $v$ is a vector in some representation $\mathbb{V}$ of $T$.
\end{itemize}
Thus,  a non-convex optimization problem may be efficiently solved by algorithms developed for convex optimization,  as long as the objective function is g-convex\cite{ZW13,ZS2016,LSCJ17}.  Guided by this principle,  the following fundamental problem arises naturally in the literature. See,  for example,  \cite{pp05},  \cite[Chapter~4]{UC1994} and \cite[Section~5.4]{Nisheeth18}.
\begin{problem}\label{prob} 
Given a function on a manifold,  can we prove the existence/non-existence of a Riemannian metric,  such that the function is g-convex?
\end{problem}
As mentioned in \cite[Section~5.4]{Nisheeth18},  it is relatively easy to verify the g-convexity of a function if the Riemannian metric is specified .  However,  Problem~\ref{prob} is much more difficult as there are infinitely many Riemannian metrics on a given manifold.  For instance,  each smooth function $\lambda$ on $\mathbb{R}$ determines a Riemannian metric $e^{\lambda(x)} dx^2$ on $\mathbb{R}$.  Some attempts were made to solve Problem~\ref{prob},  but only a few stringent sufficient conditions were obtained \cite{UC1994,pp05,pripoae2013generalized,Nisheeth18}.
\subsection*{Main results}
This paper consists of two parts.  The first part is concerned with Problem~\ref{prob} in general.  We prove two necessary conditions (cf. ~Propositions \ref{prop:necessary condition} and \ref{prop:nec-cond2}) and two sufficient conditions (cf.  Propositions~\ref{thm:sufficient0} and \ref{sufficient1}) for a smooth function to be g-convex with respect to some connection.  As a consequence of our criteria,  we are able to partially address Problem~\ref{prob} (cf.~Theorem~\ref{thm:sparseness}).  In particular,  our result reveals that under some assumptions,  \emph{most functions can not be g-convex with respect to any Riemannian metric}.  Since a connection is not necessarily associated to a (pseudo-)Riemannian metric,  we further establish criteria (cf.~Corollary~\ref{sufficient_condition_for_LC} and Proposition~\ref{suffient_necessary_condition_for_LC}) to check when this indeed occurs.  

The second part of the paper concentrates on Problem~\ref{prob} for polynomial functions on $\mathbb{R}^n$.  There are three reasons for us to restrict our considerations.  Firstly,  by Hopf-Rinow Theorem \cite[Theorem~6.13]{lee1997},  any g-convex function on a compact Riemannian manifold must be constant.  This solves Problem~\ref{prob} for compact manifolds.  Unfortunately,  even the existence of a non-constant g-convex function on a non-compact Riemannian manifold remains a mystery to differential geometers for over 50 years \cite{BO69,Wu71,CG72,Yau74},  not to mention Problem~\ref{prob}.  The former problem can be solved for manifolds of positive sectional curvature \cite{GW74,GW76},  since such a manifold is diffeomorphic to some Euclidean space \cite{GM69}.  Thus,  it is reasonable to investigate Problem~\ref{prob} for $\mathbb{R}^n$.  Secondly,  equipping $\mathbb{R}^n$ with a Riemannian metric is an instructive approach to analyze algorithms for convex optimization problems \cite{NT02,LV17,LV18}.  Thirdly,  polynomials are arguably the most commonly used functions in optimization theory and related fields \cite{Edwards85,Nesterov00,Lasserre00,Jolliffe02,CLMW11}.  Using the criteria we established in the first part,  we completely characterize g-convex univariate (resp.  quadratic,  monomial,  additively separable) polynomials in Theorem~\ref{snc_for_dim1} (resp.  \ref{condition_of_degree_two},  \ref{monomial},  \ref{snc_for_seperable_case}).  Moreover,  we estimate the density of such polynomials in an appropriate sense (cf.~Theorems~\ref{prob_dim_1},  \ref{theorem for quadratic polynomial},  \ref{thm:DD3} and \ref{snc_for_seperable_case_univariate}).  Surprisingly,  \emph{the density of such polynomials asymptotically decreases to zero,  although they consist of a full-dimensional subset}.
\section{Preliminaries}
\subsection{Notations}
First we fix some notations that will be frequently used in the sequel.  We denote 
\[
\mathsf{S}^2(\mathbb{R}^n) \coloneqq \lbrace
X \in \mathbb{R}^{n\times n}:  X^\tp = X
\rbrace,\quad 
\mathsf{S}^2_+(\mathbb{R}^n) \coloneqq \lbrace
X \in \mathsf{S}^2 (\mathbb{R}^n):  X \succeq 0
\rbrace.
\]
Let $\mathbb{R}[x_1,\dots,x_n]$ be the space of polynomials in $n$ variables.  We denote by $\mathbb{R}[x_1,\dots,x_n]_{\leq d}$ the subset of $\mathbb{R}[x_1,\dots,x_n]$ consisting of all polynomials of degree at most $d$.  We define 
\[
A_{n,d} \coloneqq \lbrace
f \in \mathbb{R}[x_1,\dots,  x_n]_{\le d}: \text{$f$ is convex with respect to some connection}
\rbrace.
\]

For a smooth manifold $\mathcal{M}$ and $x\in \mathcal{M}$,  we denote by $\mathbb{T}_x \mathcal{M}$ (resp.  $\mathbb{T} \mathcal{M}$) the tangent space (resp.  tangent bundle) of $\mathcal{M}$.  Moreover,  the space of all smooth maps from $\mathcal{M}$ to another smooth manifold $\mathcal{N}$ is denoted by $\mathcal{C}^{\infty}(\mathcal{M},  \mathcal{N})$.  If in particular that $\mathcal{N} = \mathbb{R}$,  we abbreviate $\mathcal{C}^{\infty}(\mathcal{M},  \mathcal{N})$ as $\mathcal{C}^{\infty}(\mathcal{M})$.
\subsection{Differential geometry}
We provide some basics of differential geometry in this subsection and refer interested readers to standard references \cite{Lee2003,  DC1976,  lee1997,  KN1996vol1} for more details.  
\subsubsection{Connection} Let $\mathcal{M}$ be a smooth manifold and let $\mathfrak{X}(\mathcal{M})$ be the space of vector fields on $M$.  A \emph{connection} on $\mathcal{M}$ is a smooth map 
\[
\nabla: \mathfrak{X}(\mathcal{M}) \times\mathfrak{X}(\mathcal{M}) \to \mathfrak{X}(\mathcal{M}),\quad \nabla_{X} Y \coloneqq \nabla(X,Y)
\]
such that for any  $f_1,f_2,  f\in \mathcal{C}^{\infty}(\mathcal{M})$,  $X_1,X_2,Y\in \mathfrak{X}(\mathcal{M})$ and $a\in \mathbb{R}$,  we have
\begin{align*}
\nabla_{f_1X_1+f_2X_2}Y &=f_1\nabla_{X_1}Y+f_2\nabla_{X_2}Y,\quad \nabla_{X}(aY)=a\nabla_{X}Y,  \\
\nabla_{X}(fY) &=X(f)Y+f\nabla_{X}Y,\quad \nabla_X Y-\nabla_Y X=[X,Y].
\end{align*}
Any connection $\nabla$ is naturally extended for differential forms and higher order tensor fields on $\mathcal{M}$.  Thus,  for each integer $k \ge 1$,  it makes sense to define $\nabla^{k}$ as the operator obtained by applying $\nabla$ $k$-times.
\begin{lemma}[gluing lemma]\cite[proposition4.5]{lee1997}\label{lem:gluing}
Let $\{U_i\}_{i\in I}$ be an open covering of $\mathcal{M}$ and let $\{\rho_i\}_{i\in I}$ be a partition of unity subordinate to $\{U_i\}_{i\in I}$.  Suppose that for each $i\in I$,  $\nabla_i$ is a connection on $U_i$.  Then $\nabla:=\sum \rho_i \widetilde{\nabla}_i$ is a connection on $M$,  where for each $i\in I$,  $\widetilde{\nabla}_i$ is defined by 
\[
\widetilde{\nabla}_i (X,  Y)_p = \begin{cases}
\nabla_i (X,  Y)_p &\text{~if~}p\in U_i\\
0&\text{~otherwise~}
\end{cases}.
\]
\end{lemma}
The \emph{curvature tensor} of $\nabla$ is defined as
\[
R: \mathfrak{X}(\mathcal{M})  \times  \mathfrak{X}(\mathcal{M}) \times  \mathfrak{X}(\mathcal{M}) \to \mathfrak{X}(\mathcal{M}),\quad 
R(X,Y,Z):=\nabla_X\nabla_Y Z-\nabla_X\nabla_Y Z-\nabla_{[X,Y]}Z.
\]
In particular,  for any $X,  Y \in \mathfrak{X}(\mathcal{M})$,  we have a linear map 
\[
R(X,Y,\cdot):  \mathfrak{X}(\mathcal{M})  \to \mathfrak{X}(\mathcal{M}),\quad 
Z \mapsto R(X,Y,Z).
\]
Let $\gamma: (0,1) \to \mathcal{M}$ be a smooth curve on $\mathcal{M}$.  If $\nabla_{\dot{\gamma}(t)}\dot{\gamma}(t) = 0$,  then $\gamma$ is called a \emph{geodesic},  where $\dot{\gamma}$ denotes the tangent vector field of $\gamma$.  For each $f\in \mathcal{C}^{\infty}(\mathcal{M})$,  the \emph{Hessian} of $f$ with respect to $\nabla$ is 
\[
\Hess_{\nabla} f: \mathfrak{X}(\mathcal{M})  \times  \mathfrak{X}(\mathcal{M}) \to \mathcal{C}^{\infty}(\mathcal{M}),\quad \Hess_{\nabla} f(X,Y) \coloneqq X (Y(f)) - (\nabla_X Y) (f).
\]
\subsubsection{Riemannian geometry}
If $(\mathcal{M}, \mathsf{g})$ is a pseudo-Riemannian manifold,  the \emph{Levi-Civita connection} of $\mathsf{g}$ is the unique connection $\nabla$ on $\mathcal{M}$ such that 
\[
Z(\mathsf{g}(X,Y))= \mathsf{g} ( \nabla_Z X,Y )+ \mathsf{g} (X,\nabla_Z Y)
\]
for any $X,Y,Z \in \mathfrak{X}(\mathcal{M})$.  To indicate the dependence of $\nabla$ on $\mathsf{g}$,  we decorate it as $\nabla_{\mathsf{g}}$.   Given $f\in \mathcal{C}^{\infty}(\mathcal{M})$,  the \emph{pseudo-Riemannian gradient} of $f$ is the unique $\grad_{\mathsf{g}} f \in \mathfrak{X}(\mathcal{M})$ such that 
\[
\mathsf{g}( \grad_{\mathsf{g}} f ,  X) = X(f)
\]
for any $X \in \mathfrak{X}(\mathcal{M})$.  The \emph{pseudo-Riemannian Hessian} of $f$ is $\Hess_{\nabla_{\mathsf{g}}} f$, abbreviated as $\Hess_{\mathsf{g}}f$. Moreover,  if we equip $\mathbb{R}^n$ with the standard Euclidean metric,  then we denote by $\nabla_{\mathsf{e}}$,  $\grad_{\mathsf{e}} f$ and $\Hess_{\mathsf{e}} f$ the corresponding connection,  gradient and Hessian,  respectively.  In this case,  $\grad_{\mathsf{e}} f$ and $\Hess_{\mathsf{e}} f$ coincide with their counterparts in calculus.

\subsubsection{Calculations in a local chart}
Let $\mathcal{M}$ be an $n$-dimensional manifold with a connection $\nabla$.  In a local chart $x = (x_1,\dots,  x_n):\mathbb{R}^n \to U$ of $\mathcal{M}$,  there exist $\Gamma^k_{ij} \in \mathcal{C}^{\infty}(\mathcal{M})$,  $1 \le i,j,k \le n$,  which are called the \emph{Christoffel symbols} of $\nabla$,  such that 
\[
\nabla_{\frac{\partial}{\partial x_i}} \frac{\partial}{\partial x_j} = \sum_{k=1}^n \Gamma^k_{ij}  \frac{\partial}{\partial x_k}.
\]
According to \cite[Lemma~4.4]{lee1997},  $\nabla$ is uniquely determined by $\Gamma^k_{ij}$'s.  Clearly,  $\Gamma^k_{ij} = \Gamma^k_{ji}$ for all $1 \le i,j,k, \le n$.

The \emph{curvature tensor} $R$ of $\nabla$ in $U$ is determined by 
\begin{equation}\label{eq_for_curvature1}
R \left( \frac{\partial }{\partial x_i},\frac{\partial }{\partial x_j},\frac{\partial }{\partial x_k} \right) = \sum_{l=1}^n R_{ijk}^l\frac{\partial }{\partial x_l},\quad
    R_{ijk}^l \coloneqq \frac{\partial \Gamma_{jk}^l}{\partial x_i}-\frac{\partial \Gamma_{ik}^l}{\partial x_j}+\sum_{t=1}^n (\Gamma_{jk}^t\Gamma_{it}^l-\Gamma_{ik}^t\Gamma_{jt}^l)
\end{equation}
for $1 \le i,j,k \le n$.  

For each $f\in \mathcal{C}^{\infty}(U)$,  we may represent $\Hess_{\nabla} f$ as an $n\times n$ matrix whose $(i,j)$-th element is $\Hess_\nabla f (\partial /\partial x_i, \partial/\partial x_j)
$,  $1 \le i,j \le n$.  
\begin{lemma}\cite[Section~1.3]{UC1994}\label{lem:Hessian translation}
Let $\nabla$ be a connection on an open subset $U\subseteq \mathbb{R}^n$ and let $\{\Gamma_{ij}^k: 1 \le i,j,k \le n\}$ be the corresponding Christoffel symbols.  For any $f\in \mathcal{C}^{\infty}(\mathbb{R}^n)$,  we have 
\[ 
\Hess_\nabla f \left( \frac{\partial }{\partial x_i},\frac{\partial }{\partial x_j} \right) = \frac{\partial^2 f}{\partial x_i \partial x_j} - \sum_k \Gamma_{ij}^k\frac{\partial f}{\partial x_k},\quad 1 \le i \le j \le n.
\]
\end{lemma}
If $\nabla = \nabla_{\mathsf{g}}$ for some pseudo-Riemannian metric $\mathsf{g}$,  then in $U$ we can write $\mathsf{g} =\sum_{i,j=1}^n \mathsf{g}_{ij} d x_i \otimes d x_j$ and the \emph{Riemannian gradient} of $f$ together with its derivatives can be calculated by
\begin{equation}\label{eq_LC_connection}
\grad_{\mathsf{g}} f =\sum_{i,j=1}^n  \frac{\partial f}{\partial x_i} \mathsf{g}^{ij}\frac{\partial }{\partial x_j}, \quad
    \frac{\partial \mathsf{g}_{ij}}{\partial x_k} =\sum_{l=1}^n  \Gamma_{ik}^lg_{jl}+\sum_{l=1}^n  \Gamma_{jk}^l \mathsf{g}_{il}. 
\end{equation}
where $(\mathsf{g}^{ij})_{i,j=1}^n$ is the inverse of the matrix $(\mathsf{g}_{ij})_{i,j=1}^n$.
\subsection{G-convex functions}
References for this subsection are \cite{Ab2007,BN2023,boyd2004,UC1994,rockafellar-1970a,NY2018}.  Let $\mathcal{M}$ be a manifold with a connection $\nabla$.  A subset $U\subseteq \mathcal{M}$ is called a \emph{g-convex subset} of $\mathcal{M}$ with respect to $\nabla$ if for any $x,y\in U$, the geodesic connecting $x,y$ (if it exists) is contained in $U$.  A function $f: U\to \mathbb{R}$ is \emph{g-convex} with respect to $\nabla$ if for any $x,y\in U$ with a geodesic $\gamma:[0,1]\to U$ such that $\gamma(0)= x$,  $\gamma(1) = y$,  the function $f\circ \gamma: [0,1] \to \mathbb{R}$ is convex in the usual sense.  If $\nabla = \nabla_{\mathsf{g}}$ for some pseudo-Riemannian metric $\mathsf{g}$,  then such $f$ is said to be \emph{g-convex} with respect to $\mathsf{g}$.  
%For clarity,  convex functions on $\mathbb{R}^n$ in the usual sense are called \emph{Euclidean convex} functions.  

Let $\mathcal{M}$ be a manifold with a connection $\nabla$.  The following theorem characterizes g-convex functions on $\mathcal{M}$ with respect to $\nabla$.
\begin{theorem}\cite[Subsections 3.5 \& 3.6]{UC1994}\label{thm: g-convex}
Assume $U$ is a g-convex subset of $\mathcal{M}$.  For $f \in \mathcal{C}^{\infty}(U)$,  the following are equivalent:
    \begin{enumerate}[(a)]
        \item $f$ is g-convex with respect to $\nabla$.
        \item $f(x)+\dot{\gamma}(f(x))\leq f(y)$ for any $x,y\in U$,  where $\gamma$ is the geodesic connecting $x$ and $y$.  \label{thm: g-convex:item2}
        \item $\Hess_\nabla f(x)\succeq 0$ for all $x\in U$. \label{thm: g-convex:item3}
    \end{enumerate}
\end{theorem}
%%%%%%%%%%%%%%%%%%%%%%%%%%%%%%%%%%%%%%%%%%%%%%%%%%%%%%%%%%%%%%%%%%
\section{Criteria for g-convexity}\label{sec:criteria}
This section is concerned with criteria for a function to be g-convex with respect to some connection.  We provide two necessary conditions and two sufficient conditions. 
\begin{proposition}[Necessary condition for g-convexity I]\label{prop:necessary condition}
If $f\in \mathcal{C}^{\infty}(\mathbb{R}^n)$ is g-convex with respect to some connection,  then $\Hess_{\textsf{e}}(f)$ is positive semidefinite at every critical point of $f$.
\end{proposition}
\begin{proof}
Let $\nabla$ be a connection on $\mathbb{R}^n$ with respect to which $f$ is  g-convex.  By Lemma~\ref{lem:Hessian translation},  we have 
\[
\Hess_{\textsf{e}} f \left( \frac{\partial }{\partial x_i},\frac{\partial }{\partial x_j} \right) = \frac{\partial^2 f}{\partial x_i \partial x_j} = \Hess_\nabla f \left( \frac{\partial }{\partial x_i},\frac{\partial }{\partial x_j} \right) + \sum_k \Gamma_{ij}^k\frac{\partial f}{\partial x_k},\quad 1 \le i \le j \le n.
\]
Therefore,  if $p$ is a critical point of $f$,  then $\Hess_{\textsf{e}}(f)_p = \Hess_{\nabla}(f)_p \succeq 0$ by condition~\eqref{thm: g-convex:item3} in Theorem~\ref{thm: g-convex} .
\end{proof}
%\blue{\begin{remark}
%    Even if the manifold concerned with is not $\mathbb{R}^n$, by Lemma ~\ref{lem:Hessian translation} and condition~\eqref{thm: g-convex:item3} in Theorem~\ref{thm: g-convex}, we can also get that in \emph{any} coordinate charts, the matrix $(\frac{\partial^2 f}{\partial x_i \partial x_j})$ should be positive semidefinite at \emph{any} critical points of $f$.
%\end{remark}}
Next we consider sufficient conditions.
\begin{proposition}[Sufficient condition for g-convexity I]\label{thm:sufficient0}
Let $\mathcal{M}$ be an $n$-dimensional manifold and let $f \in \mathcal{C}^{\infty}(\mathcal{M})$.  If $f$ has no critical points,  then for any symmetric $\mathcal{C}^\infty(M)$-bilinear map $A: \mathfrak{X}(\mathcal{M}) \times \mathfrak{X}(\mathcal{M}) \to \mathcal{C}^{\infty}(\mathcal{M})$,  there exists a connection $\nabla$ such that $\Hess_{\nabla}(f) = A$.  In particular,  any smooth function with no critical point is g-convex with respect to some connection.
\end{proposition}
\begin{proof}
Let $\mathsf{g}$ be a Riemannian metric on $M$ and let $\{U_\alpha \}_{\alpha \in \Lambda}$ be an open cover of $M$ such that each $U_\alpha$ admits a coordinate map $x^{\alpha}: U_\alpha \to \mathbb{R}^n$.  Here $\Lambda$ is some index set and the existence of $\mathsf{g}$ is guaranteed by \cite[Proposition~13.3]{Lee2003}.  We pick a smooth partition of unity $\{\phi_{\alpha}\}_{\alpha \in \Lambda}$ subordinate to $\{U_\alpha \}_{\alpha \in \Lambda}$ and define $A^\alpha \coloneqq A|_{U_{\alpha}}$ for each $\alpha \in \Lambda$.  By definition,  we have $\sum_{\alpha \in \Lambda}\phi_\alpha A^\alpha = A$.  We observe that if there is a  connection $\nabla^\alpha$ such that $\Hess_{\nabla^\alpha}(f|_{U_{\alpha}}) = A^{\alpha}$ for each $\alpha \in \Lambda$,  then $\nabla \coloneqq \sum_{\alpha \in \Lambda} \phi_\alpha \nabla^\alpha$ is a connection on $M$ satisfying $\Hess_{\nabla}(f) = A$ by Lemmas~\ref{lem:gluing} and \ref{lem:Hessian translation}.

Therefore,  it is sufficient to assume that $M = U \subseteq \mathbb{R}^n$ is an open subset. In this situation,  we may identify $A$ with $(a_{ij})_{i,j=1}^n \in \mathcal{C}^{\infty}(U,\mathsf{S}^2(\mathbb{R}^n))$ where
\[
a_{ij} \coloneqq A\left( \dfrac{\partial}{\partial x_i},\dfrac{\partial }{\partial x_j} \right) \in \mathcal{C}^{\infty}(U),\quad 1 \le i,  j \le n.
\]
We denote $f_i \coloneqq \partial f / \partial x_i$ for $1\le i \le n$ and consider the following system of homogeneous linear equations:
\begin{equation}\label{thm:sufficient0:Sij}
\sum_{k=1}^n f_k S_{ij}^k  =0,  \quad    S_{ij}^k - S_{ji}^k =0,\quad 1 \le i,j \le n.
\end{equation}
For each solution $\{S^k_{ij}: 1 \le i,j,k \le n\}$ of \eqref{thm:sufficient0:Sij},  we define 
\begin{equation}\label{thm:sufficient0:Gammaij}
\Gamma_{ij}^k \coloneqq \sum_{l=1}^n \frac{(f_{ij}-a_{ij}) \mathsf{g}^{kl}f_l}{\lVert \grad_{\mathsf{g}} (f) \rVert^2}+S_{ij}^k,
\end{equation}
where $f_{ij} \coloneqq \partial^2 f / \partial x_i \partial x_j$ for $1 \le i, j \le n$.  Since $f$ has no critical point,  $\lVert \grad_{\mathsf{g}}(f)(x)\rVert > 0$.  Hence $\Gamma^k_{ij}$'s are well-defined.  It is straightforward to verify that
    \begin{equation}\label{thm:sufficient0:eq}
\Gamma^k_{ij} = \Gamma^k_{ji},\quad \frac{\partial^2 f}{\partial x_i\partial x_j}-\sum_{k=1}^n \Gamma_{ij}^k \frac{\partial f}{\partial x_k}=a_{ij}\quad  1 \le i,j \le n.
\end{equation}
Let $\nabla$ be the connection determined by $\{ \Gamma^k_{ij}: 1 \le i,j,k \le n \}$.  Then Lemma~\ref{lem:Hessian translation} together with \eqref{thm:sufficient0:eq} implies $\Hess_{\nabla}(f) = A$ and this completes the proof.
\end{proof}

\begin{remark}\label{rmk:sufficient0}
Some remarks on Proposition~\ref{thm:sufficient0} are in order:
\begin{itemize}
\item[$\diamond$] Proposition~\ref{thm:sufficient0} generalizes \cite[Theorem~6.6, Chapter~3]{UC1994},  which is concerned with the case $A = 0$. 
\item[$\diamond$] If $f\in \mathcal{C}^{k}(\mathcal{M})$ and $A $ is differentiable up to order $k-2$ for some integer $k\ge 2$,  then the construction in the proof of Proposition~\ref{thm:sufficient0} provides us a connection of order $k-2$.
\item[$\diamond$] It is clear that the connection constructed in the proof of Proposition~\ref{thm:sufficient0} is not unique,  even if we take $A = 0$. 
\end{itemize}
\end{remark}

It is proved in \cite[Theorem~6]{pp05} that every $f$ with a unique critical point which is a minimum must be g-convex with respect to some connection.  This is generalized to a large extent by the proposition that follows.
\begin{proposition}[Sufficient condition for g-convexity II]\label{sufficient1}
Let $(\mathcal{M},\mathsf{g})$ be a Riemannian manifold and let $f$ be a smooth function on $\mathcal{M}$.  Suppose that for any critical point $x\in \mathcal{M}$ of $f$,  there exists a neighbourhood $U$ of $x$ such that  $\hess_{\mathsf{g}} f (y) \succeq 0$ for all $y \in U$.  Then there exists a connection $\nabla$ such that $f$ is g-convex with respect to $\nabla$.
\end{proposition}
\begin{proof}
Let $C_f$ be the set of critical points of $f$ and let $V \coloneqq \mathcal{M} \setminus C_f$.  For each $x\in C_f$,  we choose and fix an open neighbourhood $U_x$ on which $\hess_{\mathsf{g}}f$ is positive semidefinite. 

By definition,  we have $C_f = df^{-1}(0)$ where $df$ is the differential of $f$ regarded as a smooth map $df: \mathcal{M} \to \mathbb{T}^\ast \mathcal{M}$ and $0$ denotes the zero section of the vector bundle $\mathbb{T}^\ast \mathcal{M}$.  In particular,  we may conclude that $V$ is an open subset of $\mathcal{M}$.  Thus,  the family $\{U_x: x\in C_f\}\cup \{V\}$ is an open covering of $M$.  Let $\{\eta_x:  x\in C_f \} \cup \{\eta_{V}\}$ be a partition of unity  subordinate to $\{U_x: x\in C_f\}\cup \{V\}$.

On each $U_x$,  we define a connection $\nabla_x \coloneqq \nabla_{\mathsf{g}}|_{U_x}$.  Then $\hess_{\nabla_x}f(y) \succeq 0 $ for all $y\in U_x$.  Since $V$ is an open submanifold of $\mathcal{M}$ on which $f$ has no critical points,  Proposition~\ref{thm:sufficient0} implies the existence of a connection $\nabla_{V}$ on $V$ such that  $\hess_{\nabla_V}f(y) \succeq 0 $ for all $y\in V$.  By the Lemma \ref{lem:gluing},  $\{\nabla_x:  x\in C_f\}$ together with $\nabla_V$ defines an  connection $\nabla$ on $\mathcal{M}$.  By the non-negativity of $\eta_x$ and $\eta_V$,  it is straightforward to verify that $\hess_{\nabla}f (y) \succeq 0$ for every $y\in M$ and this completes the proof.
\end{proof}
\begin{example}\label{ex:necessary}
In Proposition~\ref{sufficient1},  the condition that $\hess_{\mathsf{g}} f \succeq 0$ in  \emph{a neighbourhood} of $x\in C_f$ is essential.  It is not sufficient to just require that $\hess_{\mathsf{g}} f(x) \succeq 0$.  For instance,  the function $f(x,y) =x^2y^2$ has $C_f = \{0\}\times \mathbb{R}\cup \mathbb{R}\times \{0\}$ and its Euclidean Hessian matrix is $\Hess_{\textsf{e}}  f = \begin{bsmallmatrix}
        2y^2 & 4xy\\
        4xy & 2x^2
\end{bsmallmatrix}$,  which is positive semidefinite on $C_f$ and is indefinite on $\mathbb{R}^2 \setminus C_f$.  However,  $f$ is not  g-convex with respect to any connection according to  Theorem~\ref{monomial}.  We notice that this example also indicates that the converse of Proposition~\ref{prop:necessary condition} is false.  

We also notice that this condition is not a necessary condition.  For instance,  we have $C_f = \{0\} \times \mathbb{R}$ for $f(x,y) =x^2e^y$ .  It is observed in \cite{pp05} that $f$ is  g-convex with respect to the connection defined by 
\[
\Gamma_{ij}^k = \begin{cases}
1 &\text{~if $(i,j,k)=  (1,2,1)$ or $(2,1,1)$ }, \\
0 &\text{~otherwise}.
\end{cases}
\]
Nonetheless,  its Euclidean Hessian matrix $\Hess_{\textsf{e}}  f  =\begin{bsmallmatrix}
        2e^y & 2xe^y\\
         2xe^y & x^2e^y
    \end{bsmallmatrix}$ is positive semidefinite on $C_f$ and is indefinite on $\mathbb{R}^2 \setminus C_f$.  
\end{example}
For a better illustration,  we summarize criteria and their relations discussed above in Figure~\ref{fig:criteria},  where we assume that $f$ is a smooth function on an open subset of $\mathbb{R}^n$.
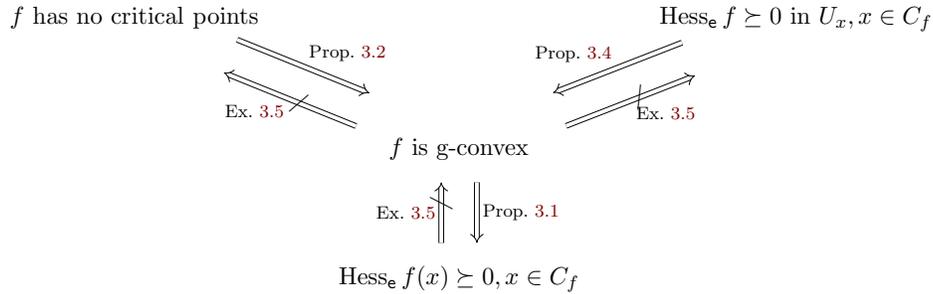
\begin{figure}[!htbp]
\adjustbox{scale=0.85,center}{
\begin{tikzcd}
	{f\text{~has no critical points}} && {\hess_{\textsf{e}} f \succeq 0\text{~in~}U_x, x\in C_f} \\
	\\
	& {f \text{~is  g-convex}} \\
	\\
	& {\hess_{\textsf{e}} f(x) \succeq 0, x \in C_f}
	\arrow["{\text{Prop.~\ref{thm:sufficient0}}}", shift left=3, shorten <=22pt, shorten >=22pt, Rightarrow, from=1-1, to=3-2]
	\arrow["{\text{Prop.~\ref{sufficient1}}}"', shift right=3, shorten <=23pt, shorten >=23pt, Rightarrow, from=1-3, to=3-2]
	\arrow["{\text{Ex.~\ref{ex:necessary}}}", "\not"{marking}, shift left=3, shorten <=22pt, shorten >=22pt, Rightarrow, from=3-2, to=1-1]
	\arrow["{\text{Ex.~\ref{ex:necessary}}}"', "\not"{marking}, shift right=3, shorten <=23pt, shorten >=23pt, Rightarrow, from=3-2, to=1-3]
	\arrow["{\text{Prop.~\ref{prop:necessary condition}}}", shift left=3, shorten <=6pt, shorten >=6pt, Rightarrow, from=3-2, to=5-2]
	\arrow["{\text{Ex.~\ref{ex:necessary}}}", "\not"{marking}, shift left=3, shorten <=6pt, shorten >=6pt, Rightarrow, from=5-2, to=3-2]
\end{tikzcd}}
\caption{Criteria for  g-convexity}
\label{fig:criteria}
\end{figure}

In the following we prove a necessary condition for a function to be  g-convex with respect to some geodesically complete connection,  which might be of independent interest.  Here a connection is \emph{geodesically complete} if every geodesic extends infinitely. 
\begin{proposition}[Necessary condition for g-convexity II]\label{prop:nec-cond2}
Suppose that $f\in \mathcal{C}^{\infty}(\mathcal{M})$ is  g-convex with respect to some geodesically complete connection.  If $C_f$ is discrete,  then $|C_f| \le 1$.
\end{proposition}
\begin{proof}
We notice that by the argument in the proof of \cite[Theorem~7.5,  Chapter~3]{UC1994},  each $c \in C_f$ is a global minimal point.  Let $f_0 \coloneqq \min_{x\in \mathcal{M}} f(x)$.  Then $C_f = \{x \in \mathcal{M}:  f(x) \le f_0\}$ is a level set.  According to \cite[Theorem~3.4, Chapter~3]{UC1994},  we conclude that $C_f$  must be g-convex.  In particular,  it is path-connected.  Since $C_f$ is discrete and $M$ is Hausdorff,  this implies that $C_f$ is either empty or is a singleton.
\end{proof}
\begin{theorem}]\label{thm:sparseness}
Let $\mathcal{M}$ be a smooth manifold.  We denote by $A(\mathcal{M})$ (resp.  $A^{\mathsf{c}}(\mathcal{M})$) the subset of $\mathcal{C}^{\infty}(\mathcal{M})$ consisting of functions that are  g-convex with respect to some  (resp.  geodesically complete ) connection.  We have:
\begin{enumerate}[(a)]
\item If $f\in \mathcal{C}^{\infty}(\mathcal{M})$ is g-convex with respect to some complete Riemannian metric and $\Vol(\mathcal{M}) < \infty$,  then $f$ is a constant.  \label{thm:sparseness-1}
\item If $\mathcal{M}$ is compact,  then $A(\mathcal{M})$ is a nowhere dense subset of $\mathcal{C}^{\infty}(\mathcal{M})$ in compact-open topology.
\label{thm:sparseness-2}
\item If $\mathcal{M} = \mathbb{R}^n$,  then there is an open dense subset $U \subsetneq \mathbb{R}[x_1,\dots,  x_n]_{\le d}$ such that any $f \in U \cap A^{\mathsf{c}}(\mathbb{R}^n)$ has at most $1$ critical point.
%the probability that $f \in C_{n}(r)$ also lies in $A^{\mathsf{c}}(\mathcal{M})$ is $O(1/n\log d)$,  where $C_n(r)$ is the cube of length $2r$ in $\mathbb{R}[x_1,\dots,  x_n]_{\le d}$ centered at the origin,  equipped with the uniform probability distribution.
\label{thm:sparseness-3}  
\end{enumerate}
\end{theorem}
\begin{proof}
\eqref{thm:sparseness-1} is proved in \cite[Proposition~2.2]{BO69}.  For \eqref{thm:sparseness-2},  we notice that if $f\in \mathcal{C}^{\infty}(\mathcal{M})$ is  g-convex with respect to some  connection with finite $C_f$,  then $f$ is a constant.   Indeed,  the compactness of $\mathcal{M}$ implies the existence of a maximizer $x_0\in \mathcal{M}$ of $f$.  Let $\gamma: [-\delta,  \delta] \to \mathcal{M}$ be a geodesic passing through $x_0 = \gamma(0)$.  Since $f \circ \gamma$ is g-convex,  we have $(f\circ \gamma(-\delta) + f\circ \gamma(\delta))/2 \ge f \circ \gamma(0)$.  The maximality of $f(x_0)$ implies $f\circ \gamma(\delta) = f\circ \gamma(-\delta) = f(x_0)$.  Since $\delta > 0$ can be arbitrarily small,  we conclude that $f$ is a constant around $x_0$,  which contradicts to the assumption that $C_f$ is finite.  Note that for a Morse function $f$ on a compact manifold, critical points are isolate\cite[Corollary 2.3,Chapter 1]{MR1963} and hence $C_f$ is finite. We recall that the set $S$ of Morse functions is dense in $\mathcal{C}^{\infty}(\mathcal{M})$ \cite[Theorem~1.2.5]{AD14},  while the set of constant functions is nowhere dense in $S$.  Hence $A(\mathcal{M})$ is nowhere dense in $\mathcal{C}^{\infty}(\mathcal{M})$.  To prove \eqref{thm:sparseness-3},  we recall that there is a Zariski open subset $U \subseteq \mathbb{R}[x_1,\dots,  x_n]_{\le d}$ such that every $f \in U$ has $(d-1)^n$ complex critical points.  Therefore,  by Proposition~\ref{prop:nec-cond2} we may conclude that $f$ has at most one critical point if $f$ also lies in $A^{\mathsf{c}}$.
\end{proof}
%%%%%%%%%%%%%%%%%%%%%%%%%%%%%%%%%%%%%%%%%%%%%%%%%%%%%%%%%%%%%%%%%%
\section{Criteria for Levi-Civita connection}\label{sec:Levi-Civita}
In Section \ref{sec:criteria},  we obtain some criteria for a function to be  g-convex with respect to some  connection.  However,  it is well-known that a connection is not necessarily the Levi-Civita connection determined by a pseudo-Riemannian metric.  Let $\nabla$ be an  connection  on $\mathcal{M}$ and let $(U;x_1,\dots,  x_n)$ be a local chart of $\mathcal{M}$.  Suppose that $\Gamma^k_{ij} \in \mathcal{C}^\infty(M),  1\le i,j,k, \le n$ are Christoffel symbols of $\nabla$ on $U$.  By definition,  $\nabla = \nabla_{\mathsf{g}}$ for some pseudo-Riemannian metric $\mathsf{g}$ if and only if $\mathsf{g}$ is a solution of \eqref{eq_LC_connection}.  We notice that \eqref{eq_LC_connection} has a trivial solution $\mathsf{g}_{ij} = 0, 1 \le i,  j \le n$.  Unfortunately,  this is not a desired solution since $\mathsf{g} = 0$ is not a pseudo-Riemannian metric on $\mathcal{M}$.  In the literature \cite{Alekseevski68,s1973,Bryant87, KN1996vol1,Merkulov99},  the existence of a pseudo-Riemannian metric $\mathsf{g}$ on $\mathcal{M}$ such that $\nabla = \nabla_{\mathsf{g}}$ is discussed in terms of the holonomy group of $\nabla$,  which is not possible to compute in general.  This section is devoted to a brief discussion on practical conditions for the existence of $\mathsf{g}$.

Let $x\in \mathcal{M}$ be a fixed point.  For each integer $k \ge 0$,  we define $\mathfrak{L}_k$ to be the Lie subalgebra of $\End(\mathbb{T}_x \mathcal{M})$ generated by linear maps
\begin{equation}\label{eq:Lk}
X_{v_1,\dots,  v_{j+2}}:  \mathbb{T}_x \mathcal{M} \to \mathbb{T}_x \mathcal{M},\quad X_{v_1,\dots,  v_{j+2}}(u) \coloneqq (\nabla^j R(v_1,v_2))(u, v_3,  \dots,  v_{j+2}),
\end{equation}
where $v_1,\dots,  v_{j +2} \in \mathbb{T}_x \mathcal{M}$ and $0 \le j \le k$.
\begin{proposition}[Stability]\label{prop:LC-criterion}
Let $\nabla$ be an analytic  connection on a connected and simply connected manifold $\mathcal{M}$ and let $x\in \mathcal{M}$ be a fixed point.  Then there exists an integer $(\dim \mathcal{M})^2 \geq k_0 \ge 0$ such that $\mathfrak{L}_{k_0} = \mathfrak{L}_{k}$ for any $k \ge k_0$.  If there exists a non-degenerate symmetric bilinear form $B: \mathbb{T}_x \mathcal{M} \times \mathbb{T}_x \mathcal{M} \to \mathbb{R}$ such that 
\[
B(Xu,  v) + B(u,  Xv)=0
\]
for any $X \in \mathfrak{L}_{k_0}$,  then $\nabla = \nabla_{\mathsf{g}}$ for some pseudo-Riemannian metric $\mathsf{g}$ on $\mathcal{M}$ of the same signature as $B$.
\end{proposition}
\begin{proof}
Denote $n \coloneqq \dim \mathcal{M}$.  By definition,  we have $\mathfrak{L}_k\subseteq \mathfrak{L}_l$ whenever $k\leq l$.  Since $\dim \End(\mathbb{T}_x \mathcal{M}) = n^2$,  there exists $1 \le k_0 \le n^2$ such that $\mathfrak{L}_{k_0+1}=\mathfrak{L}_{k_0}$.  We claim that $\mathfrak{L}_{k_0} = \mathfrak{L}_{k}$ for all $k \ge k_0$.  It suffices to prove that $\mathfrak{L}_{k}=\mathfrak{L}_{k+1}$ implies $\mathfrak{L}_{k+2}\subseteq \mathfrak{L}_{k+1}$,  for any positive integer $k$.  By \cite[Lemma~1 of Chapter~3]{KN1996vol1}),  any $X \coloneqq X_{v_1,\dots,v_{k+4}} \in \mathfrak{L}_{k+2}$ can be written as
\begin{align}\label{Lie_alg}
X &=\nabla_{v_{k+4}}((\nabla^{k+1}R(v_1,v_2))(v_3,\dots,v_{k+3})) -\nabla^{k+1}R(\nabla_{v_{k+4}}v_1,v_2)(v_3,\dots,v_{k+3}) \nonumber  \\
    &-\nabla^{k+1}R(v_1,\nabla_{v_{k+4}}v_2)(v_3,\dots,v_{k+3}) - \sum_{i=3}^{k+3}\nabla^{k+1}(R(v_1,v_2))(v_3,\dots,\nabla_{v_{k+4}}v_i,\dots,v_{k+3}). 
\end{align}
We denote $Y \coloneqq \nabla_{v_{k+4}}((\nabla^{k+1}R(v_1,v_2))(v_3,\dots,v_{k+3}))$ and prove that $Y \in \mathfrak{L}_{k+1}$.  By assumption, $(\nabla^{k+1}R(v_1,v_2))(v_3,\dots,v_{k+3})\in \mathfrak{L}_{k+1} =  \mathfrak{L}_{k}$,  hence $(\nabla^{k+1}R(v_1,v_2))(v_3,\dots,v_{k+3})$ is an $\mathbb{R}-$linear combination of endomorphisms of the form $(\nabla^{j}R(u_1,u_2))(u_3,\dots,u_{j+2}))$ for $0\leq j\leq k$.  According to \eqref{Lie_alg},  $\nabla_{v_{k+4}} ( (\nabla^{j} R(u_1,u_2)) (u_3,\dots,u_{j+2})) \in \mathfrak{L}_{k+1}$ which implies $Y \in \mathfrak{L}_{k+1}$. 

Since $\nabla$ is analytic,  \cite[Chapters 2 and 3]{KN1996vol1} implies that $\mathfrak{L} \coloneqq \bigcup_{k=0}^\infty \mathfrak{L}_{k}$ is the Lie algebra of the holonomy group of $\nabla$ at $x$.  As $\mathfrak{L}_{k_0} = \mathfrak{L}_{k}$ for any $k \ge k_0$,  we simply have $\mathfrak{L} = \mathfrak{L}_{k_0}$.  This together with \cite[Theorem]{s1973} indicates the existence of $\mathsf{g}$.
\end{proof}
We first notice that $\nabla$ in Proposition~\ref{prop:LC-criterion} is assumed to be analytic.  This assumption ensures that $\mathfrak{L}_{k_0}$ is the Lie algebra of the holonomy group of $\nabla$ at $x$.  As illustrated by the example in \cite[Section~3]{s1973},  this is not always true for smooth connections.  We also remark that the Lie algebra $\mathfrak{L}$ in the proof of Proposition~\ref{prop:LC-criterion} and \cite[Chapters 2 and 3]{KN1996vol1} is generated by infinitely many linear maps,  but it must be finitely generated since it is finite dimensional.  Proposition~\ref{prop:LC-criterion} explicitly provides a finite set of generators for $\mathfrak{L}$.  Lastly,  Proposition~\ref{prop:LC-criterion} ensures the existence of $\mathsf{g}$.  We refer the interested readers to \cite{s1973} for an explicit construction of $\mathsf{g}$.  The corollary that follows is a direct consequence of Proposition~\ref{prop:LC-criterion},  which will be useful in the sequel.

\begin{corollary}[Sufficient condition for Levi-Civita connection]\label{sufficient_condition_for_LC}
Let $\nabla$ be an analytic  connection on a connected and simply connected manifold $\mathcal{M}$ and let $x\in \mathcal{M}$ be a fixed point.  If $\mathfrak{L}_{k_0} = 0$,  then for any integer $0 \le p \le \dim \mathcal{M} \eqqcolon n$,  there is a pseudo-Riemannian metric $\mathsf{g}$ on $\mathcal{M}$ of signature $(p,n-p)$ such that $\nabla = \nabla_{\mathsf{g}}$.
\end{corollary}
We conclude this section by a simple criterion for the non-existence of $\mathsf{g}$ such that $\nabla$ is Levi-Civita with respect to $\mathsf{g}$.
\begin{proposition}[Necessary condition for Levi-Civita connection]\label{suffient_necessary_condition_for_LC}
Let $\nabla$ be a analytic  connection on a manifold $\mathcal{M}$ and let $x\in \mathcal{M}$ be a fixed point.  If $\dim{\mathfrak{L}_k}> (n^2-n)/2$ for some $k$,  then there is no pseudo-Riemannian metric $\mathsf{g}$ on $\mathcal{M}$ such that $\nabla = \nabla_{\mathsf{g}}$.
\end{proposition}
\begin{proof}
Suppose $\nabla = \nabla_{\mathsf{g}}$ for some pseudo-Riemannian metric $\mathsf{g}$.  Then for each $x\in \mathcal{M}$,  $\mathsf{g}_x$ is a non-degenerate bilinear form on $\mathbb{T}_x(\mathcal{M})$.  We denote 
\[
\mathfrak{g} \coloneqq \left\lbrace
L \in \End(\mathbb{T}_x(\mathcal{M})): B(Lu,v) + B(u,Lv) =0, \; u,v\in \mathbb{T}_x(\mathcal{M})
\right\rbrace.
\]
By \cite[Chapters 2 and 3]{KN1996vol1},  $\mathfrak{L} \coloneqq \bigcup_{k=0}^\infty \mathfrak{L}_{k} \subseteq \mathfrak{g}$.  Since $\dim \mathfrak{g} = n(n-1)/2$,  this contradicts to the assumption that $\mathfrak{L}_k > n(n-1)/2$.
\end{proof}
%\blue{
%\begin{lemma}
%    Let $\mathcal{B}$ be a non-degenerate symmetric bilinear form on an $n-$dimensional Euclidean space $V$ with signature $(p,n-p)$, $\mathfrak{G}$ be the total invariant group of $\mathcal{B}$, that is, the Lie group of invertible linear transformations $\mathcal{A}$ of $V$ such that $\mathbb{B}(\mathcal{A}x,\mathcal{A}y)=\mathcal{B}(x,y),\forall x,y\in V$ and $\mathfrak{g}$ be the Lie algebra of $\mathfrak{G}$. Then $\dim \mathfrak{g}=\frac{n^2-n}{2}$
%\end{lemma}
%\begin{proof}
%    Choose a basis of $V$ under which the matrix of $\mathcal{B}$ is of the standard form $B_{0}:=\begin{pmatrix}
%        I_p& 0\\
%        0&-I_{n-p}
%    \end{pmatrix}$
%    Then under this basis, the corresponding matrix $A$ of $\mathcal{A}\in\mathfrak{g}$ satisfes $A^tB_0+B_0A=0$. Denote $Z(B_0):=\{A\in \mathbb{R}^{n\times n}|A^tB_0+B_0A=0\}$, then it's easy to calculate that $\dim \mathfrak{g}=\dim Z(B_0)=\frac{n^2-n}{2}$
%\end{proof}}

%%%%%%%%%%%%%%%%%%%%%%%%%%%%%%%%%%%%%%%%%%%%%%%%%%%%%%%%%%%%%%%%%%
\section{g-convex univariate polynomials}
This section is devoted to a discussion on  g-convex univariate polynomials.  To begin with,  we establish the following lemma,  which will be used repeatedly in the rest of this paper.
\begin{lemma}[Local behavior of a non-negative function]\label{key lemma}
Let $f \in \mathcal{C}^{\infty}(\mathbb{R}^n)$ be a non-negative function.  Suppose moreover that $f=x_i^k g$ for some $1 \le i \le n$,  $k\in\mathbb{N}$ and $g\in \mathcal{C}^{\infty}(\mathbb{R}^n)$.  Denote $g_0 \coloneqq  \lim_{x_i\to 0}g \in \mathcal{C}^{\infty}(\mathbb{R}^{n-1})$.  Then we have:
    \begin{enumerate}[(a)]
        \item If $k$ is even, then $g_0$ is non-negative.  \label{key lemma:item1}
        \item If $k$ is odd,  then $g_0 = 0$. \label{key lemma:item2}
    \end{enumerate}
\end{lemma}
\begin{proof}
Without loss of generality,  we assume that $i = 1$ so that $f = x_1^k g$ and $g_0(x_2,\dots, x_n) = g(0,x_2,\dots, x_n)$.  If $k$ is even, then $x_1^k>0$ for all $x_1\neq 0$.  Since $f$ is nonnegative, 
\[
g_0(x_2,\dots,x_n)=\lim_{x_1\to 0}\dfrac{f(x_1,\dots,x_n)}{x_1^k}\geq 0
\] 
for any $(x_2,\dots,\dots,x_n)\in \mathbb{R}^{n-1}$ and this proves \eqref{key lemma:item1}.

We prove \eqref{key lemma:item2} by contradiction.  If $g_0\neq 0$,  then there is some $(a_2,\dots,a_n)\in \mathbb{R}^{n-1}$ such that $c:=g_0(a_2,\dots, a_n)\neq 0$.  By definition,  we have
\[
\sign(c) \lim_{x_1 \to 0}  \dfrac{f(x_1,a_2,\dots,a_n)}{x_1^k}   = |c| > \frac{|c|}{2}.
\]
Thus for a sufficiently small $\delta >0$ and any $x_1$ such that $\sign(c)  x_1 \in (-\delta,  0)$,  we may derive
\[
f(x_1,a_2,\dots,  a_n) < \frac{c x_1^k}{2} < 0.
\]
This contradicts to the non-negativity of $f$. 
\end{proof}

\begin{lemma}\label{1dim-connection}
For any smooth  connection $\nabla$ on $\mathbb{R}$,  there is a Riemmannian metric $\mathsf{g}$ such that $\nabla = \nabla_{\mathsf{g}}$.  As a consequence,  $f\in \mathcal{C}^{\infty}(\mathbb{R})$ is  g-convex with respect to some smooth connection if and only if it is g-convex with respect to some Riemannian metric. 
\end{lemma}
\begin{proof}
We observe that a Riemannian metric $\mathsf{g}$ on $\mathbb{R}$ can be written as $\mathsf{g} =  e^{\lambda(x)}dx^2$ for some $\lambda(x) \in \mathcal{C}^{\infty}(\mathbb{R})$.  Then its Christoffel symbol is simply a single function $\Gamma_{11}^1 = \lambda'$.
%\[
%\Gamma_{11}^1=\dfrac{1}{2g}\dfrac{d g}{dx}=\dfrac{1}{2}\dfrac{d(\ln g)}{dx}=\dfrac{1}{2}\dfrac{d \lambda}{dx}. 
%\]
By the existence of the solution of an ODE,  we obtain the existence of $\mathsf{g}$ such that $\nabla = \nabla_{\mathsf{g}}$.  
\end{proof}
In the following,  we will completely characterize  g-convex univariate polynomials.  According to Theorem~\ref{thm: g-convex} and Lemma ~\ref{1dim-connection},  $f\in \mathbb{R}[x]$ is  g-convex if and only if there exists some $\lambda \in \mathcal{C}^{\infty}(\mathbb{R})$ such that
\begin{equation}\label{eq2_dim1}
d \coloneqq b'+\dfrac{1}{2}b \lambda' \geq 0
\end{equation}
where $b \coloneqq f'$. We decompose $b$ as $b(x)=a\prod_{i=1}^{k}(x-u_i)^{r_i}c(x)$, where $u_1< \cdots <u_k$ are the distinct real roots of $b(x)$ with multiplicities $r_1,\dots, r_k$ respectively,  $a$ is a non-zero real number and $c$ is a monic polynomial with no  real root.  We adopt the convention that $b$ has no real root if $k = 0$.  Then 
\begin{equation}\label{d_frac}
    b'(x)=a\sum_{i=1}^{k}r_i(x-u_i)^{r_i-1}\prod_{j\neq i}(x-u_j)^{r_j}c(x)+a\prod_{i=1}^{k}(x-u_i)^{r_i}c'(x)
\end{equation}

\begin{lemma}\label{lemma_two}
Let $f,  b,  k$ and $r_1,\dots,  r_k$ be as above.  If $k \ge 1$ and $f$ is  g-convex with respect to some smooth connection,  then $r_i$ is odd for all $1 \le i \le k$.
\end{lemma}
\begin{proof}
    If $b$ has a real root $u$ with even multiplicity $2r \ge 2$,  then we can write $b(x)=(x-u)^{2r}p(x)$ for some polynomial $p$ such that $p(u)\neq 0$.  Since $f$ is  g-convex,  \eqref{eq2_dim1} implies
\[
        d(x)=(x-u)^{2r-1}\left( 2rp(x)+(x-u)p'(x)+\dfrac{1}{2}(x-u)p(x)\lambda'(x) \right) \geq 0
\]   
By Lemma \ref{key lemma},  we have $2rp(u)=0$,  but this contradicts the assumption that $p(u)\neq 0$ and $r>0$.
\end{proof}

\begin{lemma}\label{lemma_one}
Let $f,  b,  k$ and $r_1,\dots,  r_k$ be as above.  If $k \ge 1$ and $f$ is  g-convex with respect to some smooth connection,  then there is at most one $1 \le i \le k$ such that $r_i$ is odd.
\end{lemma}
\begin{proof}
    Suppose on the contrary that there exist $1\leq p<q\leq k$ such that $r_p$ and $r_q$ are odd.  Without loss of generality,  we assume that 
    \[
    q=\max\{1 \le j \le k: r_j\text{~is odd}\},\quad p=\max\{1 \le j \le q-1: r_j\text{~is odd}\}.
    \]
By \eqref{eq2_dim1} and \eqref{d_frac},  we have 
\[
d(x)=a\prod_{i=1}^k(x-u_i)^{r_i-1}\left( \sum_{i=1}^kr_ic(x)\prod_{j\neq i}(x-u_j)+\prod_{i=1}^k(x-u_i)(c'(x)+\dfrac{1}{2}c(x)\lambda(x)) \right)\geq 0.
\]
Since both $r_p$ and $r_q$ are odd,  Lemma \ref{key lemma} implies  
\[
ac(u_p)r_p\prod_{l\neq p}(u_p-u_l)^{r_l}\geq 0,\quad ac(u_q)r_q\prod_{l\neq q}(u_q-u_l)^{r_l}\geq 0.
\]
By multiplying the two inequalities,  we obtain 
\[
a^2c(u_p)c(u_q)r_pr_q \prod_{i > p}(u_p-u_i)^{r_i} \prod_{j> q}(u_q-u_j)^{r_j} \geq 0,
\]
since $u_1< \cdots<u_k$.  We notice that $c(x)$ has no real root,  hence $c(u_p)$ and $c(u_q)$ have the same sign.  This implies $c(u_p)c(u_q)>0$.  Moreover,  by $a\ne 0$ and $r_i\geq 1 $ for all $1 \le i \le k$,  we may derive   
\begin{equation}\label{lemma_one:eq1}
\prod_{i>p}(u_p-u_i)^{r_i} \prod_{j > q}(u_q-u_j)^{r_j} > 0.
\end{equation}
By the choice of $p$ and $q$,  $r_i$ is even whenever $i > p$ and $i \neq q$.  Therefore,  \eqref{lemma_one:eq1} can be further simplified to     $(u_p-u_q)^{r_q}>0$,  which contradicts to the assumption that $u_p<u_q$ and $r_q$ is odd. 
\end{proof}

\begin{theorem}[g-convex univariate polynomials]\label{snc_for_dim1}
Let $f \in \mathbb{R}[x]$ be a non-constant univariate polynomial.  The following are equivalent: 
\begin{enumerate}[(a)]
\item $f$ is  g-convex  with respect to some connection.\label{snc_for_dim1:item11}
\item $f$ is g-convex  with respect to some Riemannian metric.\label{snc_for_dim1:item22}
\item One of the following holds:
    \begin{enumerate}[(i)]
        \item $f$ has no critical point. \label{snc_for_dim1:item1}
        \item $f'(x)=(x-u)^{2r-1}p(x)$ for some real number $u$,  integer $r\ge 1$ and $p \in \mathbb{R}[x]$ such that $p(x) > 0$.\label{snc_for_dim1:item2}
    \end{enumerate}
\end{enumerate}
\end{theorem}
\begin{proof}
Denote $b(x) \coloneqq f'(x)$.  The equivalence between \eqref{snc_for_dim1:item11} and \eqref{snc_for_dim1:item22} follows from Lemma~\ref{1dim-connection}.  If $f$ is  g-convex with respect to some smooth connection and $b(x)$ has a root $u\in \mathbb{R}$,  then Lemmas~\ref{lemma_two} and \ref{lemma_one} imply that $b(x)=(x-u)^{2r-1} p(x)$ where $r\ge 1$ and $p(x) \in \mathbb{R}[x]$ has no real root.  Moreover,  by \eqref{eq2_dim1} we have
\[
(x-u)^{2r-2}\left( (2r-1)p(x)+(x-u)p'(x)+\dfrac{1}{2}(x-u)p(x)\lambda'(x))\right) \geq 0.
\]  
Since $2r- 2\ge 0$ is even,  Lemma \ref{key lemma} indicates that $p(u) \geq 0$.  Therefore,  $p$ is positive everywhere.

Conversely,  if  \eqref{snc_for_dim1:item1} holds,  then by Proposition~\ref{thm:sufficient0},  there exists an  connection on $\mathbb{R}$ such that $f$ is  g-convex.  Next,  we assume that \eqref{snc_for_dim1:item2} holds.  In this case,  we have $f''(x) = b'(x) =(x-u)^{2r-2}((2r-1) p(x)+(x-u)p'(x))$.  We claim that $f''(x)$ is non-negative in a neighborhood of $u$ so that Proposition~\ref{sufficient1} applies.  Indeed,  we notice that 
\[
\lim_{x\to u}\dfrac{b'(x)}{(x-u)^{2r-2}} = (2r-1) p(u)>0,
\]
from which we conclude that $b'(x)/(x-u)^{2r-2} > 0$ in a neighborhood of $u$.  This further implies that $b'(x) \ge 0$ in a neighborhood of $u$. 
\end{proof}

\begin{remark}
Theorem~\ref{snc_for_dim1} completely characterizes univariate polynomials which are  g-convex with respect to some smooth  connection.  We make several observations below:
\begin{itemize}
    \item[$\diamond$] If \eqref{snc_for_dim1:item1} (resp.  \eqref{snc_for_dim1:item2}) holds,  then $\deg(f)$ is odd (resp.  even). 
     \item[$\diamond$] According to the proof of Theorem~\ref{snc_for_dim1},  if \eqref{snc_for_dim1:item2} holds,  then $f$ is convex in a neighborhood of $u$.
    \item[$\diamond$] If $f$ is convex,  then \eqref{snc_for_dim1:item1} may be refined.  Indeed,  since $b' = f{''} \ge 0$,  \eqref{snc_for_dim1:item1} implies that $f$ is a linear function.
\end{itemize}
\end{remark}
\begin{example}\label{counter}
We consider the cubic polynomial $f(x)=x^3$.  By Theorem~\ref{snc_for_dim1},  it is clear that $f$ is not  g-convex with respect to any smooth  connection.  We also observe that $f$ satisfies the necessary condition given in Proposition~\ref{prop:necessary condition},  indicating again (cf.  Example~\ref{ex:necessary}) that the converse of Proposition~\ref{prop:necessary condition} is false.
\end{example}

Next,  for an integer $d \ge 0$,  we denote by $\mathbb{R}[x]_{\le d}$ the space of univariate polynomials of degree at most $d$ and set 
\begin{align*}
A_{1,d} &\coloneqq \lbrace
f \in \mathbb{R}[x]_{\le d}: f \text{~is  g-convex with respect to some  connection}
\rbrace,  \\
D_{d} &\coloneqq \{b\in \mathbb{R}[x]_{\leq d}: b = f' \text{~for some~}f\in A_{1,d+1} \}.
%&= \{b\in \mathbb{R}[x]_{\leq d}: b \text{~satisfies conditions \eqref{snc_for_dim1:item1} or \eqref{snc_for_dim1:item2} in Theorem \ref{snc_for_dim1}}\}.
\end{align*}
Let $\partial:\mathbb{R}[x]_{\leq d}\rightarrow \mathbb{R}[x]_{\leq d-1}$ be the $\mathbb{R}$-linear map defined by $\partial f(x) \coloneqq f'(x)$.  Obviously,  we have $ \partial(A_{d+1})= D_{d}$ and $\partial^{-1}(b)\cong \mathbb{R}$ for all $b\in D_d$.  This together with the definition of topological dimension leads to the lemma  that follows.
\begin{lemma}\label{relation_of_DCP_and_GCP}
We have $\dim A_{1,d+1} =\dim D_d +1$.
\end{lemma}
By \eqref{eq2_dim1},  it is clear that both $A_{1,d}$ and $D_{d}$ are cones in the vector space $\mathbb{R}[x]_{\le d}$.  However,  neither of them is g-convex.  For instance,  although $f_1(x)=x^2+4x+5,  f_2(x)=2x-5 \in D_{2}$,  it is easy to verify that $f_1(x)/2 + f_2(x)/2 =x(x+3) \not\in D_{2}$.

Let $M_d$ be the space consisting of monic complex polynomials of degree $d$.  We consider the map $\Phi: \mathbb{C}^{d} \to M_d$ defined by $\Phi(a_1,\dots,  a_d) \coloneqq \prod_{j=1}^d (x - a_j)$.
\begin{lemma}\label{nondegenerate}
The Jacobian matrix $J(\Phi)$ of $\Phi$ is invertible at $x = (x_1,\dots,  x_d)\in \mathbb{C}^d$ if and only if $x_1,\dots,  x_d$ are distinct.  In particular,  $\Phi$ is a local diffeomorphism at such point. 
\end{lemma}
\begin{proof}
We consider the isomorphism $\tau: M_d  \to \mathbb{C}^d$ defined by 
\[
\tau \left( x^d + \sum_{j=0}^{d-1} c_j x^j \right) \coloneqq (-c_{d-1}, c_{d-2},\dots,  (-1)^{d-1}c_1,  (-1)^d c_0).
\]
We observe that $\varphi \coloneqq \tau \circ \Phi: \mathbb{C}^d \to M_d  \to \mathbb{C}^d$ is simply $\varphi (x) = (\sigma_1(x),\dots,  \sigma_d(x))$,
where $\sigma_k$ is the $k$-th elementary symmetric function for $1 \le k \le d$.  According to \cite[Theorem2.1]{WuXingYuan},  $\det(J(\varphi)(x)) =\prod_{i<j}(x_i-x_j)$,  from which the lemma follows immediately as $\tau$ is an isomorphism.  The fact that $\Phi$ is a local diffeomorphism is a direct consequence of the inverse function theorem.  
\end{proof}

\begin{proposition}[Dimension I]\label{dim_of_DCP}
For any integer $d\ge 0$,  we have $\dim A_{1,d} = d+1$.
\end{proposition}
\begin{proof}
Lemma~\ref{relation_of_DCP_and_GCP} implies that $\dim D_{d-1} = \dim A_{1,d}-1$.  By definition,  it is sufficient to find some $h \in D_{d-1}$ and prove that $h$ has a neighborhood $W$ in $\mathbb{R}[x]_{\le d-1}$ such that $W \subseteq D_{d-1}$ and $\dim W = d$.  We split the discussion with respect to the parity of $d$. 

If $d-1$ is even,  then Theorem~\ref{snc_for_dim1} implies that $h(x) \coloneqq x^{d-1}+1$ is contained in $D_{d-1}$.  Moreover,  $h$ has $d$ distinct roots $\zeta_1,\dots,  \zeta_{d-1} \in \mathbb{C} \setminus \mathbb{R}$.  By Lemma \ref{nondegenerate},  there exists a neighborhood $U_1$ of $(\zeta_1,\dots,  \zeta_{d-1})$ in $\mathbb{C}^{d-1}$ such that $\Phi|_{U_1}$ is a diffeomorphism.  Thus,  $V_1 \coloneqq \Phi(U_1)$ is a neighborhood of $h =\Phi(\zeta_1,\dots,  \zeta_{d-1})$ in $M_{d-1}$.  Furthermore,  $W_1 \coloneqq V_1 \cap M_{d-1}(\mathbb{R})$ is a non-empty neighborhood of $h$ in $M_{d-1}(\mathbb{R})$,  where $M_{d-1}(\mathbb{R})$ is the space consisting of monic real polynomials of degree $d-1$.  We may shrink $U_1$,  if necessary,  so that every $f \in W_1$ has no real root.  Lastly,  we define 
\[
W \coloneqq \left\lbrace f= \sum_{j=0}^d c_j x^j:  c_d > 0,  c_d^{-1}f \in W_1 \right\rbrace.
\]
It is straightforward to verify that $W \simeq  W_1 \times (0,\infty)$ has dimension $d$ and $W \subseteq \mathbb{R}[x]_{\le d-1}$ is a neighborhood of $h$. 

If $d-1$ is odd,  then we let $W'$ be the neighborhood of $x^{d-2} + 1$ in $\mathbb{R}[x]_{\le d-2}$ as above,  and we define
\[
W \coloneqq  \{ (x-u)q(x):  u\in \mathbb{R},\; q\in W' \}. 
\]
Clearly,  $W$ is a neighborhood of $h(x) \coloneqq x(x^{d-2} + 1)$ in $\mathbb{R}[x]_{\le d-1}$.  We notice that any $f\in W'$ is non-negative.  By Theorem \ref{snc_for_dim1},  we have $W \subseteq D_{d}$.  Moreover,  $W \simeq W' \times \mathbb{R}$ has dimension $d$ and this completes the proof.
\end{proof}

For each positive real number $r$,  we observe that
\begin{equation}\label{eq:Pd}
P_d(r) \coloneqq \left\lbrace
f = \sum_{j=0}^d a_j x^j \in \mathbb{R}[x]_{\leq d}:  |a_j| \le r,\;  0 \le j \le d
\right\rbrace \simeq [-r,r]^{d+1}.
\end{equation}
Thus we may equip $P_d(r)$ with the uniform distribution.  
As an application of Proposition~\ref{littlewood} and Theorem~\ref{snc_for_dim1},  we may derive an estimate of the density of  g-convex polynomials in $P_d(r)$.  We recall a classical result. 
\begin{proposition}\cite[Theorem 2]{Littlewood1938}\label{littlewood}
The probability that $f \in P_d(1)$ has $O(\log d/(\log\log d)^2)$ real roots is $O(1/\log d)$.
\end{proposition}
\begin{theorem}[Density I]\label{prob_dim_1}
The probability that $f \in P_d(r)$ is  g-convex with respect to some  connection is $O(1/\log d)$. 
\end{theorem}
\begin{proof}
By \eqref{eq:Pd},  it is sufficient to consider $P_d(1)$.  We notice that $\lim_{d\to \infty} \log d/(\log\log d)^2 = \infty$.  According to Proposition \ref{littlewood},  the probability that $f \in P_d(1)$ has at most one real root is $O(1/\log d)$.  This together with Theorem~\ref{snc_for_dim1} completes the proof.
\end{proof}
%%%%%%%%%%%%%%%%%%%%%%%%%%%%%%%%%%%%%%%%%%%%%%%%%%%%%%%%%%%%%%%%%%
\section{g-convex quadratic polynomials} 
Next we consider  g-convex quadratic polynomials.  Given a quadratic polynomial 
\begin{equation}\label{representation of a quadratic polynomials}
    f(x)=\dfrac{1}{2}x^\tp Ax+ b^\tp x +c \in \mathbb{R}[x_1,\dots,  x_n]_{\le 2},
\end{equation}
where $A \in \mathsf{S}^2(\mathbb{R}^n)$,  $c\in \mathbb{R}$ and $x,  b \in \mathbb{R}^n$ are column vectors.  Then 
\begin{equation}\label{eq:quad poly}
\grad_{\textsf{e}} f=Ax+b, \quad \Hess_{\textsf{e}}  f =A.
\end{equation}
\begin{lemma}\label{lem:quad-g-convex}
A quadratic polynomial $f \in \mathbb{R}[x_1,\dots,  x_n]_{\le 2}$ parametrized as in \eqref{representation of a quadratic polynomials} is  g-convex with respect to some connection if and only if $A \succeq 0$ or $\rank \left( \begin{bmatrix}
A & b
\end{bmatrix}\right) > \rank (A)$. 
\end{lemma}
\begin{proof}
If $f$ is  g-convex with respect to some connection and $\rank \left( \begin{bmatrix}
A & b
\end{bmatrix}\right) = \rank (A)$,  we must have $A\succeq 0$ by Proposition~\ref{prop:necessary condition} and \eqref{eq:quad poly}.  Conversely,  if $A\succeq 0$,  then $f$ is g-convex with respect to the Euclidean metric.  If $\rank \left( \begin{bmatrix}
A & b
\end{bmatrix}\right) > \rank (A)$,  then \eqref{eq:quad poly} implies that $f$ has no critical point and $f$ is  g-convex with respect to some connection by Proposition~\ref{thm:sufficient0}.
\end{proof}
Next we characterize quadratic polynomials that are g-convex with respect to some pseudo-Riemannian metric.  According to Lemma~\ref{lem:quad-g-convex},  it is sufficient to consider $f\in \mathbb{R}[x_1,\dots,  x_n]_{\le 2}$ with $ \rank \left( \begin{bmatrix}
A & b
\end{bmatrix}\right) > \rank (A)$.
\begin{lemma}[Normal form]
If $f \in \mathbb{R}[x_1,\dots,  x_n]_{\le 2}$ with $ \rank\left( \begin{bmatrix}
A & b
\end{bmatrix}\right)>  \rank(A) \eqqcolon r $,  then there exist $Q \in \O_n(\mathbb{R})$ and $v\in \mathbb{R}^n$ such that 
% a coordinate change such that $Y=QX+v$ with $Q$ an orthogonal matrix such that after the transformation, 
\begin{equation}\label{normal_form}
    f ( Q^\tp (y-v) )=\sum_{i=1}^r \mu_i y_i^2+\sum_{j=r+1}^n \nu_j y_j+\kappa
\end{equation}
where $\kappa,  \mu_i,\nu_j \in\mathbb{R}$,  $1 \le i \le r < j \le n$ are constants such that $\mu_i\neq 0$ and $\nu_{j_0} \ne 0$ for all $1 \le i \le r$ and some $r + 1 \le j_0 \le n$.
\end{lemma}
\begin{proof}
By the theory of bilinear forms\cite[Section 4~Chapter X]{GWlinearal1975},  there exists $Q\in \O_n(\mathbb{R})$ such that 
$f (Q^\tp z)=\sum_{i=1}^r \mu_i z_i^2+\sum_{j=1}^n b'_{j}z_j+c$ for some $b' \coloneqq (b'_1,\dots, b'_n) \in \mathbb{R}^n$ and $\mu_i\neq 0, i=1,\dots, r$.  Denote 
\[
v \coloneqq \left( \frac{b'_1}{2\mu_1},\frac{b'_2}{2\mu_2},\dots,\frac{b'_r}{2\mu_r}, \underbrace{0,\dots,0}_{(n-r)\text{~copies}} \right)^\tp,  \quad \kappa=c-\sum_{i=1}^r\frac{{b'_i}^2}{4\mu_i},\quad \nu_j=b'_j,  \quad r +1 \le j \le n.
\]
It is straightforward to verify that $f(Q^\tp (y-v))$ can be written as \eqref{normal_form}.  We notice that $F(y) \coloneqq f(Q^\tp (y-v))$ has no critical point,  since $f$ has no critical point.  Hence 
\[
\grad_{\textsf{e}} F(y) = \diag(2 \mu_1 y_1,\dots,  2\mu_r y_r,  \nu_{r+1},\dots,  \nu_{n}) = 0 
\]
has no solution,  which implies $\nu_{j_0} \ne 0$ for some $r +1 \le j_0 \le n$.
\end{proof}

\begin{lemma}\label{lem:quadpoly}
Let $0 \le p \le n$ be an integer.  If $f \in  \mathbb{R}[x_1,\dots,  x_n]_{\le 2}$ has no critical point,  then $f$ is g-convex with respect to a pseudo-Riemannian metric of signature $(p,n-p)$.
\end{lemma}
\begin{proof}
Without loss of generality,  we assume that
\[
    f ( x )=\sum_{i=1}^r \mu_i x_i^2+\sum_{j=r+1}^n \nu_j x_j+\kappa,
\]
where $\kappa,  \mu_i,\nu_j \in\mathbb{R}$,  $1 \le i \le r < j \le n$ are constants such that $\nu_{n} \ne 0$ and $\mu_i\neq 0$ for all $1 \le i \le r$.  
We claim that there exist $\Gamma_{lm}^k$'s such that the corresponding connection $\nabla$ satisfies
\begin{equation}\label{lem:quadpoly:eq2}
R = 0,\quad \Hess_{\nabla} f = 0.
\end{equation}
Here $R$ denotes the curvature tensor of $\nabla$.  Then the Lie subalgebra $\mathcal{L}_k$ defined in \eqref{eq:Lk} is zero for all $k\in \mathbb{N}$ and Corollary~\ref{sufficient_condition_for_LC} implies that $\nabla = \nabla_{\mathfrak{g}}$ for some pseudo-Riemannian metric $\mathsf{g}$ on $\mathbb{R}^n$ of signature $(p,n-p)$.  Therefore,  it is left to prove the existence of the desired $\Gamma^{k}_{lm}$'s.

We equip $\mathbb{R}^n$ with the Euclidean metric.  By \eqref{thm:sufficient0:eq},  the Christoffel symbols of a connection $\nabla$ on $\mathbb{R}^n$ such that $\Hess_{\nabla} f = 0$ are determined by 
%the proof of Theorem~\ref{thm:sufficient0} and \eqref{thm:sufficient0:Gammaij},  the Christoffel symbols of an  connection $\nabla$ on $\mathbb{R}^n$ with $\Hess_{\nabla} f = 0$ should satisfy the equations,  which in this case turns to 
\begin{equation}\label{lem:quadpoly:eq1}
\begin{aligned}
    \Gamma^k_{ij} - \Gamma^k_{ji}  &=0,\\
    -2\sum_{k=1}^r\Gamma_{st}^k\mu_kx_k-\sum_{k=r+1}^n\Gamma_{st}^k\nu_k &=0,\\
    2\mu_i-2\sum_{k=1}^r\Gamma_{ii}^k\mu_kx_k-\sum_{k=r+1}^n\Gamma_{ii}^k\nu_k &=0,
\end{aligned}
\end{equation}
where $1 \le i,j,k \le n$ and $(s,t) \not\in \{(1,1),\dots, (r,r)\}$.  We look for a solution of \eqref{lem:quadpoly:eq1} in the following form:
\begin{equation}\label{setting}
    \Gamma_{ij}^k = \begin{cases}
0 &\text{~if~} 1 \le k \le r\\
\text{constant to be determined} &\text{~otherwise~}
\end{cases},
\end{equation}
where $1 \le i,  j \le n$.  Then the first derivative of each $\Gamma_{ij}^k$ vanishes and \eqref{eq_for_curvature1} becomes
\[
    R_{stk}^l=\sum_{u=r+1}^n(\Gamma_{tk}^u\Gamma_{su}^l-\Gamma_{sk}^u\Gamma_{tu}^l),\quad 1 \le l,s,t,k \le n.
\]
Combining \eqref{lem:quadpoly:eq1},  \eqref{setting} and the assumption that $R = 0$,  we obtain
\begin{equation}\label{alg_eq_for_degree_two_1}
\Gamma_{ij}^p  = \sum_{u=r+1}^n(\Gamma_{ij}^u\Gamma_{lu}^k-\Gamma_{lj}^u\Gamma_{iu}^k) = \mu_p-\sum_{u=r+1}^{n}\Gamma_{pp}^u \nu_{u} = \sum_{u =r+1}^n \Gamma_{st}^u \nu_u = 0
\end{equation}
for $1 \le i, j, k, l \le n$,  $1 \le p \le r$ and $(s,t) \not\in \{(1,1),\dots,  (r,r)\}$.  Let $\Gamma_{pp}^u = a^u_p \in \mathbb{R}$ be a solution of the third equation in \eqref{alg_eq_for_degree_two_1} for $1 \le p \le r < u \le n$.  For each $1 \le i,j,k \le n$,  we define
\[
\Gamma^{k}_{ij} \coloneqq \begin{cases}
0 &\text{~if~} 1 \le k \le r,  \; 1 \le i,  j \le n \\
0 &\text{~if~} r +1 \le k \le n,\; (i,  j) \not\in \{(1,1),\dots,  (r,r)\}  \\
a^{k}_i &\text{~otherwise~}
\end{cases}.
\]
By construction,  $\{\Gamma^{k}_{ij}\}_{i,j,k=1}^n$ is a solution of \eqref{alg_eq_for_degree_two_1} and this completes the proof.
\end{proof}

\begin{example}\label{eg_for_connection}
We recall that the curvature of the Levi-Civita connection obtained in the proof of Lemma~\ref{lem:quadpoly} is zero.  It is obviously not true that every Levi-Civita connection has zero curvature.  However,  there also exists an  connection with non-zero curvature,  which is not Levi-Civita.  For instance,  we consider $f(x_1,x_2) = x_1^2+x_2 \in \mathbb{R}[x_1,x_2]_{\le 2}$.  Let 
\begin{equation}\label{eg_for_connection:eq1}
\Gamma_{lm}^k = \begin{cases}
0 &\text{if~} l \ne m \text{~or~} l = m = 1\\
\frac{4x_1}{1 + 4x_1^2} &\text{if~} l = m = 1,\;  k= 1\\
\frac{2}{1 + 4x_1^2} &\text{if~} l = m = 1,\;  k= 2\\
1 &\text{if~} l = m = 2,\; k = 1\\ 
-2x_1  &\text{if~} l = m = k = 2\\ 
\end{cases}.
\end{equation}
 It is straightforward to verify that $\Gamma_{lm}^k$'s satisfy \eqref{lem:quadpoly:eq1}. 

%\[
%a_{lm}=0,\quad S_{lm}^k = \begin{cases}
%0 &\text{if~} l \ne m \text{~or~} l = m = 1\\
%1 &\text{if~} l = m = 2,\; k = 1\\ 
%-2x_1  &\text{if~} l = m = k = 2\\ 
%\end{cases}
%\]
%for $1 \le k,l,m \le 2$. 
By a direct calculation,  we may further obtain 
\begin{align*}
X_{\frac{\partial }{\partial x_1},\frac{\partial }{\partial x_2}}
&=\begin{bmatrix}
        \frac{-2}{1+4x_1^2}&\frac{4x_1}{1+4x_1^2}\\
        \frac{4x_1}{1+4x_1^2}&\frac{-8x_1^2}{1+4x_1^2}
    \end{bmatrix}, \quad X_{\frac{\partial }{\partial x_1},\frac{\partial }{\partial x_2},\frac{\partial }{\partial x_1}}
 =\begin{bmatrix}
        \frac{8x_1}{(1+4x_1^2)^2}&\frac{4}{(1+4x_1^2)^2}\\
        \frac{-16 x_1^2}{(1+4x_1^2)^2}&\frac{-8x_1}{(1+4x_1^2)^2}
    \end{bmatrix},   \\
        X_{\frac{\partial }{\partial x_1},\frac{\partial }{\partial x_2},\frac{\partial }{\partial x_2}}
& = \begin{bmatrix}
        \frac{4x_1}{1+4x_1^2}&\frac{2}{1+4x_1^2}\\
        \frac{-8x_1^2}{1+4x_1^2}&\frac{-4x_1}{1+4x_1^2}
    \end{bmatrix},  \quad
    X_{\frac{\partial }{\partial x_1},\frac{\partial }{\partial x_2},\frac{\partial }{\partial x_1},\frac{\partial }{\partial x_1}} =     \begin{bmatrix}
        \frac{8-8x_1-96x_1^2-32x_1^3}{(1+4x_1^2)^3}&\frac{-64x_1}{(1+4x_1^2)^3}\\
        \frac{-32x_1+96 x_1^3-128x_1^5}{(1+4x_1^2)^3}&\frac{-8+8x_1+96x_1^2+32x_1^3}{(1+4x_1^2)^2}
    \end{bmatrix}.
\end{align*}
Here $X_{v_1,\dots, v_{k+2}} \in \End \left( \mathbb{T}_{(x_1,x_2)} \mathbb{R}^2 \right)$ is the linear operator defined in \eqref{eq:Lk} for $v_1,\dots,  v_{k+2} \in \mathbb{T}_{(x_1,x_2)} \mathbb{R}^2 = \spa_{\mathbb{R}} \left\lbrace \frac{\partial }{\partial x_1},  \frac{\partial }{\partial x_2} \right\rbrace$,  expressed as a matrix.  In particular,  we conclude that at $(1,0) \in \mathbb{R}^2$,  
\[
\mathcal{L}_2 = \End \left( \mathbb{T}_{(1,0)} \mathbb{R}^2 \right).
\]
Thus,  the connection defined by \eqref{eg_for_connection:eq1} is not Levi-Civita, by Proposition~\ref{suffient_necessary_condition_for_LC}.
\end{example}
Combining Lemmas~\ref{lem:quad-g-convex} and \ref{lem:quadpoly},  we are able to completely characterize all g-convex quadratic polynomials.
\begin{theorem}[g-convex quadratic polynomials]\label{condition_of_degree_two}
Let $f \in \mathbb{R}[x]_{\le 2}$ be a quadratic polynomial parametrized as in \eqref{representation of a quadratic polynomials}. The following are equivalent: 
\begin{enumerate}[(a)]
\item $f$ is  g-convex with respect to some  connection.
\item $f$ is g-convex with respect to some pseudo-Riemannian metric.
\item Either $A \succeq 0$ or $\rank \left( \begin{bmatrix}
A & b
\end{bmatrix}\right) > \rank (A)$.
\end{enumerate}
\end{theorem}

In the rest of this subsection,  we consider the dimension and density of 
\[
A_{n,2} \coloneqq \lbrace
f \in \mathbb{R}[x_1,\dots,x_n]_{\le 2}: f\text{~is  g-convex with respect to some  connection}
\rbrace.
\]
We recall from \eqref{representation of a quadratic polynomials} that 
\[
\mathbb{R}[x_1,\dots,  x_n]_{\le 2} \simeq \mathsf{S}^2(\mathbb{R}^n) \times \mathbb{R}^{n+1},\quad f(x) = x^\tp A x + b^\tp x + c \mapsto (A,b,c).  
\]
In the sequel,  we identify subsets of $\mathbb{R}[x_1,\dots,  x_n]_{\le 2}$ with their images in $\mathsf{S}^2(\mathbb{R}^n) \times \mathbb{R}^{n+1}$ under this identification.  By Theorem~\ref{condition_of_degree_two},  we have 
    \begin{equation}\label{case for quadratic polynomial}
     \mathsf{S}_+^2(\mathbb{R}^n) \times \mathbb{R}^{n+1}
   \subseteq A_{n,2} \subseteq  (\mathsf{S}_+^2(\mathbb{R}^n) \cup \Det_n)  \times \mathbb{R}^{n+1}.
    \end{equation}
We will also need the following elementary lemma.
\begin{lemma}\footnote{This proof is given by Robert Bryant on \url{https://mathoverflow.net/questions/164487/what-it-is-the-volume-of-the-unit-ball-section-of-the-cone-of-positive-definite}.}
\label{volumn of PSD}
    Let $B$ be the unit ball in $\mathsf{S}^2(\mathbb{R}^n)$ with respect to the Euclidean metric.  Then 
\[        
\Vol(B\cap \mathsf{S}_+^2(\mathbb{R}^n))=2^{-\binom{n+1}{2}}\Vol(B).
\]
\end{lemma}
\begin{proof}
We consider the map
\[
F: B \rightarrow  \mathsf{S}_+^2(\mathbb{R}^n) \cap B,\quad F(A)=\dfrac{A+I_n}{2}.
\]
It is clear that $F$ is well-defined,   and invertible.  A straightforward calculation implies that the Jacobian matrix of $F$ is simply $1/2 I_{N}$ where $N \coloneqq \dim \mathsf{S}^2(\mathbb{R}^n) = \binom{n+1}{2}$.  Therefore,  we obtain 
\[
        \Vol( \mathsf{S}_+^2(\mathbb{R}^n) \cap B) = \int_{ \mathsf{S}_+^2(\mathbb{R}^n) \cap B }\omega_{\textsf{e}}=\int_{B}F^*\omega_e=2^{-\binom{n+1}{2}}\int_{B}\omega_{\textsf{e}}= 2^{-\binom{n+1}{2}}\Vol(B)
\]
where $\omega_{\textsf{e}}$ is the volume form of $ \mathsf{S}^2(\mathbb{R}^n) \simeq \mathbb{R}^N$.
\end{proof}

\begin{theorem}[Dimension and density II]\label{theorem for quadratic polynomial}
For any positive integer $n$ and positive real number $r$,  we have 
\[
        \dim(A_{n,2})=\binom{n+2}{2},\quad  \frac{\Vol( A_{n,2} \cap C(r))}{\Vol(C(r))}= 2^{-\binom{n+1}{2}},
\] 
where $C(r) \coloneqq \{(A, \alpha)\in \mathsf{S}^2(\mathbb{R}^n) \times \mathbb{R}^{n+1}: \lVert A \rVert_{\mathsf{e}} \leq r,\; \lVert \alpha \rVert_{\mathsf{e}}\leq r \}$.  In other words,  the probability that $x^\tp A x + b^\tp x + c$ where $(A,b,c)\in C(r)$ is  g-convex with respect to some connection is $2^{-\binom{n+1}{2}}$,  if $C(r)$ is equipped with the uniform probability distribution.
\end{theorem}
\begin{proof}
It is sufficient to prove for $C \coloneqq C(1)$.  According to \eqref{case for quadratic polynomial},  we have 
\begin{align*}
\binom{n+1}{2} + (n+1 ) &\le \dim (A_{n,2}) \le \max\left\lbrace \binom{n+1}{2},  \dim \Det_n \right\rbrace +  (n+1 ),  \\
\Vol\left( \left( \mathsf{S}_+^2(\mathbb{R}^n) \times \mathbb{R}^{n+1}\right) \cap C \right) &\le \Vol\left( A_{n,2} \cap C \right) \le  \Vol\left( \left( \left( \mathsf{S}_+^2(\mathbb{R}^n) \cup \Det_n\right) \times \mathbb{R}^{n+1} \right) \cap C \right),
\end{align*}
where $\Det_n \coloneqq \lbrace
X \in \mathsf{S}^2 (\mathbb{R}^n):  \det(X) = 0
\rbrace$.  We notice that $ \Det_n$ is a hypersurface in $\mathsf{S}^2(\mathbb{R}^n)$,  thus $\dim \Det_n = \binom{n+1}{2}  - 1$ and $\Vol(\Det_n) = 0$.  This implies that 
\[
\dim (A_{n,2}) = \binom{n+1}{2} + (n+1 ) = \binom{n+2}{2}.
\]
and $\Vol\left( A_{n,2} \cap C \right) \le  \Vol\left( \left( \left( \mathsf{S}_+^2(\mathbb{R}^n) \cup \Det_n\right) \times \mathbb{R}^{n+1} \right) \cap C \right) =  \Vol\left(  \left( \mathsf{S}_+^2(\mathbb{R}^n)  \times \mathbb{R}^{n+1} \right) \cap C \right)$.  By Lemma~\ref{volumn of PSD},  we have 
\[
 \Vol\left(  \left( \mathsf{S}_+^2(\mathbb{R}^n)  \times \mathbb{R}^{n+1} \right) \cap C \right) = 2^{-\binom{n+1}{2}} \Vol(B_1) \Vol(
B_2) = 2^{-\binom{n+1}{2}} \Vol(C), 
\]
where $B_1$ (resp.  $B_2$) is the unit ball in $\mathsf{S}^2(\mathbb{R}^n)$ (resp.  $\mathbb{R}^{n+1}$).
\end{proof}
%%%%%%%%%%%%%%%%%%%%%%%%%%%%%%%%%%%%%%%%%%%%%%%%%%%%%%%%%%%%%%%%%%
\section{g-convex monomials}
Given a monomial $f \in \mathbb{R}[x_1,\dots,  x_n]_{\le d}$,  we may write 
\begin{equation}\label{eq:monomial}
f(x_1,\dots,   x_n) = ax_1^{d_1}\cdots x_n^{d_n} 
\end{equation}
for some $a\in \mathbb{R}\setminus \{0\}$ and integers $d_1,\dots,  d_n \ge 0$ such that $d_1 + \cdots + d_n \leq d$. Here we may assume $d_i > 0$ for each $1 \le i \le n$.  In fact,  if $d_1,\dots,  d_p > 0 = d_{p+1} = \cdots = d_n$ for some $1 \le p < n$ and $f$ is  g-convex as a function on $\mathbb{R}^p$ with respect to some connection $\nabla$,  then it is also  g-convex as a function on $ \mathbb{R}^n = \mathbb{R}^p \times \mathbb{R}^{n-p}$ with respect to the connection induced by $\nabla$ and the Euclidean metric on $\mathbb{R}^{n-p}$.
\begin{lemma}\label{case1_monomial}
If $f$ is a monomial as in \eqref{eq:monomial} with $d_j \ge 1$,  $1 \le j \le n$,  such that $f$ is  g-convex with respect to some  connection and $d_i \ge 2$ for some $1 \le i \le n$,  then $d_i$ is even.   
\end{lemma}
\begin{proof}
Without loss of generality,  we may assume $d_1\geq 2$.  Suppose on the contrary that $d_1$ is odd.  Let $\nabla$ be the connection such that $f$ is  g-convex and $\Gamma^k_{ij}$ be the Christoffel symbol of $\nabla$,  $1 \le i,j,k \le n$.  By Theorem~\ref{thm: g-convex} and Lemma~\ref{lem:Hessian translation}, we have 
\[
\Hess_{\nabla}f = \Hess_{\mathsf{e}} f+\sum_{k=1}^n \frac{\partial f}{\partial x_k} B^{k} \succeq 0,
\]
where $B^k$ is the matrix-valued function whose $(i,j)$-th element is $-\Gamma^k_{ij}$,  $1 \le i,j,k \le n$.  Thus,  the $(1,1)$-th diagonal element of $\Hess_{\nabla} f $ is non-negative,  i.e.,   
\[
v(x_1,\dots,  x_n) \coloneqq \frac{\partial^2 f}{\partial x_1^2}-\sum_{k=1}^n \frac{\partial f}{\partial x_k}\Gamma_{11}^k = ad_1(d_1-1)x_1^{d_1-2}x_2^{d_2}\dots x_n^{d_n} - a\sum_{k=1}^n d_kx_k^{d_k-1}\Gamma_{11}^k\prod_{l\neq k}x_{l}^{d_l} \ge 0 
\]
for any $(x_1,\dots,x_n)\in\mathbb{R}^n$.  In particular,  we have
\begin{align*}
v(x_1,1,\dots,  1) &= ad_1(d_1-1)x_1^{d_1-2} - ad_1 \gamma_1(x_1) x_1^{d_1-1} - a\sum_{k=2}^nd_k\gamma_k(x_1) x_1^{d_1} \\
&= x_1^{d_1-2}\left( ad_1(d_1-1) - ad_1 \gamma_1(x_1)  x_1 - a\sum_{k=2}^nd_k \gamma_k(x_1) x_1^2 \right) \\
& \geq 0
\end{align*}
for any $x_1\in\mathbb{R}$,  where $ \gamma_j (x_1) \coloneqq \Gamma_{11}^j(x_1,1,\dots,1)$,  $1 \le j \le n$.  Since $d_1-2$ is odd,  Lemma~\ref{key lemma} implies $ad_1(d_1-1)=0$,  which is absurd since $a\neq 0$ and $d_1>1$. 
\end{proof}

\begin{lemma}\label{case2_monomial}
Suppose $a\in \mathbb{R}\setminus \{0\}$ and $q_{t+1},\dots,  q_n$ are positive integers.  If the monomial $f=a\prod_{i=1}^t x_i\prod_{j=t+1}^n x_{j}^{2q_j}$ is  g-convex with respect to some   connection,  then $t\leq 1$.
\end{lemma}
\begin{proof}
If $t\geq 2$,  then $(0,0,1,\dots,1)$ is a critical point of $f$ since 
\[
\grad_{\textsf{e}} f = \left( a\prod_{i=2}^t x_i \prod_{j=t+1}^n x_{j}^{2q_j},ax_1 \prod_{i=3}^t x_i\prod_{j=t+1}^n x_{j}^{2q_j},\dots,2q_nax_n^{2q_n-1} \prod_{i=1}^t x_i\prod_{j=t+1}^{n-1} x_{j}^{2q_j} \right).
\]
However,  the $2 \times 2$ leading principal minor of $\Hess_{\textsf{e}} f$ at $(0,0,1,\dots,1)$ is
\begin{align*}
\det\left( \begin{bmatrix}
        \frac{\partial^2 f}{\partial x_1^2}(0,0,1,\dots,1) & \frac{\partial^2 f}{\partial x_1\partial x_2}(0,0,1,\dots,1)  \\
        \frac{\partial^2 f}{\partial x_1\partial x_2}(0,0,1,\dots,1) & \frac{\partial^2 f}{\partial x_2^2}(0,0,1,\dots,1)
    \end{bmatrix} \right) 
    = a^2\det\left( \begin{bmatrix}
        0&1\\1&0
    \end{bmatrix} \right)  = -a^2 < 0.
\end{align*}
This contradicts Proposition~\ref{prop:necessary condition}.
\end{proof}

\begin{lemma}\label{case3_monomial}
Suppose $a\in \mathbb{R}\setminus \{0\}$ and $q_{1},\dots,  q_n$ are positive integers.  If the monomial $f=ax_1\prod_{j=2}^n x_{j}^{2q_j}$ is  g-convex with respect to some   connection,  then $n\leq 2$.
\end{lemma}
\begin{proof}
Let $\nabla$ be the connection such that $f$ is  g-convex.  Assume on the contrary that $n\geq 3$.  We consider the determinant of the $2 \times 2$ lower right submatrix of $\Hess_{\nabla} f$: 
\[
        D \coloneqq \det\left( \begin{bmatrix}
            \frac{\partial^2 f}{\partial x_{n-1}^2}-\sum_{k=1}^n\frac{\partial f}{\partial x_k}\Gamma_{n-1,n-1}^{k}&\frac{\partial^2 f}{\partial x_{n-1}\partial x_n}-\sum_{k=1}^n\frac{\partial f}{\partial x_k}\Gamma_{n-1,n}^{k}\\
            \frac{\partial^2 f}{\partial x_{n-1}\partial x_n}-\sum_{k=1}^n\frac{\partial f}{\partial x_k}\Gamma_{n,n-1}^{k}&\frac{\partial^2 f}{\partial x_{n}^2}-\sum_{k=1}^n\frac{\partial f}{\partial x_k}\Gamma_{n,n}^{k}
        \end{bmatrix} \right).  
\]
%By Theorem~\ref{thm: g-convex},  $D$ is a non-negative function on $ \mathbb{R}^n$.  In particular,  $D_0(x_{n-1},x_n) \coloneqq D(1,  \dots,1,x_{n-1},x_n)$ is a non-negative function on $\mathbb{R}^2$.\blue{Denote $P:=(1,\dots,1,x_{n-1},x_n)$ and $\gamma_{ij}^k:=\Gamma_{ij}^k(P)$. Through direct calculation, we get 
%\[
%(\dfrac{\partial^2 f}{\partial x_{n-1}^2}-\sum_{k=1}^n\frac{\partial f}{\partial x_k}\Gamma_{n-1,n-1}^{k})(P)=ax_{n-1}^{2q_{n-1}-2}x_{n}^{2q_n}(\zeta_1x_n^2-2q_n\gamma_{n-1,n-1}^nx_{n-1}^2x_n),
%\]
%\[
%(\dfrac{\partial^2 f}{\partial x_n^2}-\sum_{k=1}^n\frac{\partial f}{\partial x_k}\Gamma_{n,n}^{k})(P)=ax_{n-1}^{2q_{n-1}-2}x_{n}^{2q_n}(\zeta_2x_{n-1}^2-2q_{n-1}\gamma_{n,n}^{n-1}x_{n-1}x_n^2),
%\]
%\[
%(\dfrac{\partial^2 f}{\partial x_{n-1}\partial x_n}-\sum_{k=1}^n\frac{\partial f}{\partial x_k}\Gamma_{n-1,n}^{k})(P)=(\dfrac{\partial^2 f}{\partial x_{n}\partial x_{n-1}}-\sum_{k=1}^n\frac{\partial f}{\partial x_k}\Gamma_{n,n-1}^{k})(P)
%=a\zeta_3x_{n-1}^{2q_{n-1}-1}x_{n}^{2q_n-1},
%\]
%where 
%\[\zeta_1=2q_{n-1}(2q_{n-1}-1)-x_{n-1}^2\gamma_{n-1,n-1}^1-\sum_{k=2}^{n-2}2q_kx_{n-1}^2\gamma_{n-1,n-1}^k-2q_{n-1}x_{n-1}\gamma_{n-1,n-1}^{n-1},\]
%\[\zeta_2=2q_{n}(2q_{n}-1)-x_{n}^2\gamma_{n,n}^1-\sum_{k=2}^{n-2}2q_kx_{n}^2\gamma_{n,n}^k-2q_{n}x_{n}\gamma_{n,n}^{n},\]
%\[\zeta_3=4q_{n-1}q_{n}-x_{n-1}x_n\gamma_{n-1,n}^1-\sum_{k=2}^{n-2}2q_kx_{n-1}x_n\gamma_{n-1,n}^k-2q_{n-1}\gamma_{n-1,n}^{n-1}x_n-2q_{n}\gamma_{n-1,n}^{n}x_{n-1},\]
%So we can get 
By a direct calculation,  we obtain
\[
    D_0(x_{n-1},x_n) =a^2x_{n-1}^{4q_{n-1}-4}x_n^{4q_n-4}(\eta_1x_{n-1}^2x_{n}^2+\eta_2x_{n-1}x_{n}^4+\eta_3x_{n-1}x_n^4)
\]
for some $\eta_i\in \mathcal{C}^{\infty}(\mathbb{R}^2)$,  $i=1,2,3$ such that $\eta_1(0,0)=4q_nq_{n-1}(1-2q_n-2q_{n-1})<0$.  Denote $D_1 \coloneqq \eta_1 x_n^2x_{n-1}^2+\eta_2x_{n-1}^4x_n+\eta_3x_{n-1}x_n^4$.  Then it is clear that $D_1$ is a non-negative on $\mathbb{R}^2$.

For a fixed $x_{n} \neq 0$,   we view $D_1$ as a function of $x_{n-1}$ and write $D_1 = x_{n-1} (\eta_1 x_n^2  x_{n-1} + \eta_2 x_n  x_{n-1}^3 + \eta_3 x_n^4)$.  Lemma~\ref{key lemma} indicates that $\eta_3(0,x_n) =0$ for any fixed $x_n \neq 0$.  By continuity,  we conclude that $\eta_3(0,0)=0$ . For $(x_{n-1},  x_n) \in \mathbb{R}^2$  such that $x_{n-1}x_n\neq 0$,  we notice that 
\[
    \frac{D_1}{x_{n-1}^2x_n^2}=\eta_1+\frac{\eta_2}{x_{n}}x_{n-1}^2+\frac{\eta_3}{x_{n-1}}x_n^2.
\]
Thus for each $x_n \ne 0$,  we have 
\[
 \lim_{x_n\to 0}    \lim_{x_{n-1}\to 0}\frac{D_1}{x_{n-1}^2x_n^2} =  \lim_{x_n\to 0} \left( \eta_1(0,x_n)+x_n^2\cdot\frac{\partial \eta_3}{\partial x_{n-1}}(0,x_n) \right) = \eta_1(0,0) < 0.
\]
This contradicts to the fact that $D_1$ is a non-negative function. 
\end{proof}
As a consequence of Lemmas~\ref{case1_monomial}--\ref{case3_monomial},  we may characterize monomials that are  g-convex with respect some connection.

\begin{theorem}[g-convex monomials]\label{monomial}
Given a monomial $f=ax_1^{d_1}x_2^{d_2}\dots x_n^{d_n} $ where $a \neq 0$ and $d_1,\dots,  d_n$ are non-negative integers ,  the following are equivalent:
\begin{enumerate}[(a)]
\item $f$ is  g-convex with respect to some   connection.  \label{monomial:item1}
\item $f$ is convex.  \label{monomial:item2}
\item $f=ax_j^{d_j}$ for some $1 \le j \le n$,  where $a>0$ and $d_j$ is even. \label{monomial:item3}
\end{enumerate}
\end{theorem}
\begin{proof}
Implications that \eqref{monomial:item2} $\implies$ \eqref{monomial:item1} and \eqref{monomial:item3} $\implies$ \eqref{monomial:item2} are trivial.  Thus,  it is sufficient to prove \eqref{monomial:item1} $\implies$ \eqref{monomial:item3}.  Suppose that $f$ is  g-convex with respect to a connection $\nabla$ with the Christoffel symbols $\Gamma^k_{ij}$,  $1 \le i,  j ,  k \le n$.  Then Lemmas \ref{case1_monomial}--\ref{case3_monomial} indicate that either $f=ax_i x_j^{d_j}$ or $f = a x_j^{d_j}$ for some $a \ne 0$,  $1 \le i,  j \le n$ and positive even integer $d_j$.  First we want to prove that $f=ax_i x_j^{d_j}$ is not possible.  To achieve the goal,  we may,  without loss of generality,  assume that $n = 2$ and $f=ax_1 x_2^{2p}$,  $a \ne 0,  p \ge 1$. 

By a direct calculation,  we obtain  
\[
        \Hess_{\nabla} f = \begin{bsmallmatrix}
            -ax_2^{2p}\Gamma_{11}^1-2pax_1x_2^{2p-1}\Gamma_{11}^2&2pax_2^{2p-1}-ax_2^{2p}\Gamma_{12}^1-2pax_1x_2^{2p-1}\Gamma_{12}^2 \\
            2pax_2^{2p-1}-ax_2^{2p}\Gamma_{12}^1-2pax_1x_2^{2p-1}\Gamma_{12}^2&2p(2p-1)ax_1x_2^{2p-2}-ax_2^{2p}\Gamma_{22}^1-2pax_1x_2^{2p-1}\Gamma_{22}^2  
                    \end{bsmallmatrix}.  
\]
Theorem~\ref{thm: g-convex} implies that $\Hess_{\nabla} f \succeq 0$,  thus $\det( \Hess_\nabla f ) \ge 0$.  Since 
\scriptsize{\[
\det( H_\nabla(f)) =   a^2x_2^{4p-4}\left[ (x_2^2\Gamma_{11}^1+2px_1x_2\Gamma_{11}^2)((2p-4p^2)x_1+x_2^2\Gamma_{22}^1+2px_1x_2\Gamma_{22}^2)-(-2px_2+x_2^2\Gamma_{12}^1+2px_1x_2\Gamma_{12}^2)^2 \right],
\]}\normalsize

we may conclude that 
\[
B\coloneqq (x_2^2\Gamma_{11}^1+2px_1x_2\Gamma_{11}^2)((2p-4p^2)x_1+x_2^2\Gamma_{22}^1+2px_1x_2\Gamma_{22}^2)-(-2px_2+x_2^2\Gamma_{12}^1+2px_1x_2\Gamma_{12}^2)^2\geq 0.
\]
For any fixed $x_1 \neq 0$,  we regard $B$ as a function of $x_2$.  Then Lemma~\ref{key lemma} implies $\Gamma_{11}^2(x_1,0)=0$.  So $\lim_{x_2\to 0} \Gamma_{11}^2(x_1,x_2)/x_2 =( \partial \Gamma_{11}^2/\partial x_2 ) (x_1,0)$.
For $x_2\neq 0$,  we consider $B/x^2_2$.  Since $B$ is non-negative, and $p \ge 1$,  we obtain a contradiction: $0\leq \lim_{x_1 \to 0} \lim_{x_2\to 0}\frac{B(x_1,x_2)}{x_2^2}= -4p^2$.  Thus $f=ax_j^{d_j}$ for some even $d_j$.  Since $f$ is univariate and  geodesic,  we conclude that $a > 0$ by Theorem \ref{snc_for_dim1}.
\end{proof}

Given positive integers $n,d$ and a positive real number $r$,  we denote 
\begin{align*}
M_{n,d}(r) &\coloneqq \left\lbrace a\prod_{k=1}^n x_k^{d_k} \in \mathbb{R}[x_1,\dots,  x_n]_{\le d}: |a| \le r, \; 0 \le d_1,\dots, d_n \le n\;,\sum_{k=1}^n d_k\leq d \right\rbrace ,\\ 
N_{n,d} &\coloneqq \left\lbrace (d_1,\dots,d_n)\in \mathbb{Z}^n: 0 \le d_1,\dots, d_n \le n,  \; \sum_{k=1}^n d_k\leq d \right\rbrace.
\end{align*}
By definition $M_{n,d}(\infty) \coloneqq \bigcup_{r > 0} M_{n,d}(r)$ is the space of all monomials of degree at most $d$.  Clearly,  we have $M_{n,d}(r)\setminus \{0\} \cong \left(  [-r,r]\setminus \{0\} \right) \times N_{n,d}$.  Thus,  we may equip $M_{n,d}(r)$ with the uniform probability distribution.  
\begin{theorem}[Dimension and Density III]\label{thm:DD3}
For any positive integers $n,  d$,  we have 
\[
\dim (M_{n,d}(\infty) \cap A_{n,d}) = 1.
\]
Moreover,   for any positive real number $r$,  the probability that $f \in M_{n,d}(r)$ is  g-convex with respect to some connection is $\frac{n(\lfloor \frac{d}{2}\rfloor+1)}{2\binom{n+1+d}{n}}$.
\end{theorem}
\begin{proof}
By Theorem \ref{monomial},  we have 
\[
A_{n,d} \cap ( M_{n,d}(r)\setminus \{0\} )  = \left\lbrace ax_{i}^{2j}:  1 \le i\le n,\; 0 \le j \le  \left\lfloor \frac{d}{2} \right\rfloor, r \ge a > 0\right\rbrace \simeq (0,r] \times \left\lbrace 0 ,1,\dots, \left\lfloor \frac{d}{2} \right\rfloor\right\rbrace^n.
\]
In particular,  $A_{n,d} \cap M_{n,d}(\infty)$ has dimension $1$.  Since $|N_{n,d}| = \binom{n+1+d}{n}$,  we obtain that 
\[
\frac{\Vol(A_{n,d} \cap M_{n,d}(r))}{\Vol(M_{n,d}(r))} = \frac{n(\lfloor \frac{d}{2}\rfloor+1)}{2\binom{n+1+d}{n}}.
\]
\end{proof}

%%%%%%%%%%%%%%%%%%%%%%%%%%%%%%%%%%%%%%%%%%%%%%%%%%%%%%%%%%%%%%%%%%
\section{g-convex additively separable functions}
A function $f\in \mathcal{C}^{\infty} (\mathbb{R}^n)$ is \emph{additively separable} if there exists a partition of variables
\[
\{x_1,\dots,x_n\} = \bigsqcup_{t=1}^s P_t
\]
such that $f=\sum_{t = 1}^s  f_t$ for some $f_t \in  \mathcal{C}^{\infty}(\mathbb{R}^{P_t})$,  $1 \le t \le s$.   Here $ \mathcal{C}^{\infty}(\mathbb{R}^{P_t})$ denotes the space of smooth functions whose variables are contained in the subset $P \subseteq \{x_1,\dots,  x_n\}$.  By permuting variables,  we may assume that for each $1 \le t \le s$,  $P_t =\{x_l:  l\in I_t\}$ where 
\begin{equation}\label{eq:partition}
I_t \coloneqq \{p_t+1,p_t+2,\dots,p_{t}+n_t \} \subseteq \mathbb{N},\quad n_t = \# P_t,\quad 0= p_1<p_2 < \dots < p_s < n. 
\end{equation}
Therefore,  it is sufficient to consider additively separable polynomials of such type.

Let $\{1,\dots,  n\} = \bigsqcup_{t=1}^s I_t$ be a partition as in \eqref{eq:partition}.  Suppose that $\mathbb{R}^n = \bigoplus_{t=1}^s \mathbb{R}^{n_t}$ is the decomposition of $\mathbb{R}^n$ according to the partition.  For any pseudo-Riemannian metric $\mathsf{g}_t$ on $\mathbb{R}^{n_t}$,  $1 \le t \le s$,  we denote by $\mathsf{g} \coloneqq \bigoplus_{t=1}^s \mathsf{g}_t$ the induced pseudo-Riemannian metric on $\mathbb{R}^n$.
\begin{lemma}\label{lem:ASpoly}
Let $\Gamma_{ij}^k$ be the Christorff symbols of $\mathsf{g} $.  Then $\Gamma_{ij}^k=0$ if $i,j,k$ are pairwise not in the same $I_t$ for any $1 \le t \le s$.
\end{lemma}
\begin{proof}
By definition,  we have 
\[
         \Gamma_{ij}^k=\sum_{v=1}^n \frac{1}{2}g^{kv} \left( \frac{\partial g_{iv}}{\partial x_j}+\frac{\partial g_{jv}}{\partial x_i}-\frac{\partial g_{ij}}{\partial x_v} \right) \quad 1\le  i,j,k \le n.
\]
where the matrix $(g^{ij})_{i,j=1}^n$ is the inverse of $(g_{ij})_{i,j=1}^n$.  Since $\mathsf{g}  = \bigoplus_{t=1}^s \mathsf{g}_t$,  $(g_{ij})_{i,j=1}^n = \diag(G_{1},\dots,G_{s})$ for some $G_{t}\in \mathbb{R}^{n_t\times n_t}$,  $1 \le t \le s$.  Suppose $k\in I_t$.  Then $i,  j \not\in I_t$ and $g_{ij}=g_{iv}=g_{jv}= g^{ku}= 0$ for any $v\in I_t$ and  $u\not\in I_t$.  Thus,  we obtain
\[
\Gamma_{ij}^k=\sum_{v\in I_t}\dfrac{1}{2}g^{kv}\left( \frac{\partial g_{iv}}{\partial x_j}+\frac{\partial g_{jv}}{\partial x_i} \right)=0.
\]
\end{proof}

\begin{theorem}[g-convex additively separable functions]\label{snc_for_seperable_case}
An additively separable function $f = \sum_{t=1}^s f_t \in \mathcal{C}^{\infty}(\mathbb{R}^n)$ is  g-convex with respect to some   connection if and only if one of the following holds:
       \begin{enumerate}[(a)]
           \item $f$ has no critical point.\label{snc_for_seperable_case:item1}
           \item For each $1 \le t \le s$,  $f_t$  is  g-convex with respect to some  connection. \label{snc_for_seperable_case:item2}
       \end{enumerate}
Moreover,  if for each $1 \le t \le s$,  $f_t$ is g-convex with respect to some pseudo-Riemmanian metric,  then so is $f$.
   \end{theorem}
\begin{proof}
The last assertion follows immediately from Lemma~\ref{lem:ASpoly}. Without loss of generality,  we suppose that the partition of variables for $f = \sum_{t=1}^s f_t$ is given as in \eqref{eq:partition}.  Assume that $f$ is  g-convex with respect to a connection $\nabla$ with Christorff symbols $\Gamma_{ij}^k$ and $f$ has a critical point $u=(u_1,\dots,u_n)\in\mathbb{R}^n$.  Since 
\[
\grad_{\textsf{e}} f= ( \grad_{\textsf{e}} f_1,\dots,  \grad_{\textsf{e}}f_s),
\]
we have $\grad_{\mathsf{e}} f_t (u^t)=0$,  where $u^t \coloneqq (u_{p_t+1},\dots,u_{p_t+n_t})$,  $1 \le t \le s$.
    
By Lemma~\ref{lem:Hessian translation} and Theorem~\ref{thm: g-convex},  we have
\[
\Hess_\nabla f = \diag(\Hess_{\mathsf{e}} f_1(x^1),\dots,  \Hess_{\mathsf{e}} f_t(x^t) )  - \sum_{k=1}^n \frac{\partial f}{\partial x_k}B^k(x) \succeq 0,
\] 
where for each $1 \le t \le s$ and $1 \le k \le n$,  $x^{t} \coloneqq (x_{p_t+1},\dots,x_{p_t+n_t})$ and $B^k \coloneqq (\Gamma_{ij}^k)_{i,j=1}^n$.  In particular,  the function 
\begin{align*}
H(x_1,\dots,  x_{n_1}) &\coloneqq \Hess_{\nabla} f (x_1,\dots,  x_{n_1} ,  u_2,\dots,u_n) \\
&= \diag(\Hess_{\mathsf{e}} f_1(x^1), \Hess_{\mathsf{e}} f_2(u^2), \dots,  \Hess_{\mathsf{e}} f_t(u^t) ) - \sum_{k=1}^{n_1} \frac{\partial f_1}{\partial x_k} B^k(x^1,u^2,\dots, u^t)
\end{align*}
is an $n\times n$ positive semi-definite matrix for all $(x_1,\dots, x
_{n_1}) \in \mathbb{R}^{n_1}$.  Let $H_{n_1}$ be the upper left $n_1 \times n_1$ submatrix of $H$.  Then 
\[
0 \preceq H_{n_1} = \Hess_{\mathsf{e}} f_1 - \sum_{k = 1}^{n_1}\frac{\partial f_1}{\partial x_k} B^k_{n_1},
\]
where $B^k_{n_1}$ is the upper left $n_1 \times n_1$ submatrix of $B^k$,  $1 \le k \le n_1$.  By construction of $B^k_{n_1}$,  this implies that $f_1$ is  g-convex with respect to some connection on $\mathbb{R}^{n_1}$.   Similarly,  we may prove that $f_2,\dots,  f_s$ are  g-convex with respect to some  connection,  respectively. 

Conversely,  if \eqref{snc_for_seperable_case:item1} holds,  then Proposition~\ref{thm:sufficient0} ensures that $f$ is  g-convex with respect to some connection.  Suppose that \eqref{snc_for_seperable_case:item2} holds.  According to Theorem~\ref{thm: g-convex},  for each $1 \le t \le s$,  there exist Christorff symbols $\widetilde{\Gamma}_{ij}^k\in \mathcal{C}^{\infty}(\mathbb{R}^{n_t})$ where $i,j,k\in I_t$ such that for any $y\in \mathbb{R}^{n_t}$, 
    \begin{equation}
H_t(y) \coloneqq \Hess_{\mathsf{e}} f_t (y) - \sum_{v=1}^{n_t}\frac{\partial f_t}{\partial y_v}(y) \widetilde{B}^{v} \succeq 0,
    \end{equation}
   where $\widetilde{B}^v \coloneqq (\widetilde{\Gamma}_{ij}^{v + p_t}), i,j\in I_t,   1\le  v \le n_t$.  We define 
\[
        \Gamma_{ij}^k(x^1,\dots,x^s) \coloneqq \begin{cases}
            \widetilde{\Gamma}_{ij}^k(x^t) &\quad i,j,k\in I_t,\; 1\le t \le s \\
            0&\quad \text{otherwise}
        \end{cases}
\] 
where $x^{t} \coloneqq (x_{p_t+1},\dots,x_{p_t+n_t})$,  $1 \le t \le s$.  Let $\nabla$ be the  connection defined by $\Gamma_{ij}^k$'s.  Then clearly we have 
\[
\hess_{\nabla} f (x^1,\dots,x^s)= \diag(H_1(x^1),\dots, H_s(x^s)) \succeq 0
\]
and this completes the proof.
\end{proof}
\begin{remark}
We have the following observations about Theorem~\ref{snc_for_seperable_case}.
\begin{itemize}
\item[$\diamond$] For an additively separable function $f = \sum_{t=1}^s f_t$,  it is clear that
\[
    \Hess_{\textsf{e}} f (x^1,\dots,x^s) = \diag(\Hess_{\textsf{e}} f_1 (x^1),\dots, \Hess_{\textsf{e}}  f_s (x^s)).
\]
Therefore,  $f$ is convex if and only if each $f_t$ is g-convex.  However,  Theorem~\ref{snc_for_seperable_case} indicates that for an  g-convex additively separable function,  either all its summands are also  g-convex,  or the function has no critical point.  Example~\ref{projection of the higher dimension} below shows that the latter situation indeed occurs,  in which case summands are not necessarily  g-convex.

\item[$\diamond$] In general,  for each $1 \le t \le s$,  $\mathbb{R}^{n_t}$ is a totally geodesic submanifold of $\mathbb{R}^n$ equipped with the Euclidean metric.  This ensures the equivalence between the g-convexity of an additively separable function and that of its summands.  However,  if we equip $\mathbb{R}^n$ with an arbitrary Riemannian metric,  then $\mathbb{R}^{n_t}$ is not necessarily a totally geodesic submanifold anymore and this results in \eqref{snc_for_seperable_case:item1} of Theorem~\ref{snc_for_seperable_case}. 

%\item[$\diamond$] For univariate polynomials (cf.~Theorem~\ref{snc_for_dim1}),  quadratic polynomials (cf.~Theorem~\ref{condition_of_degree_two}) and monomials (cf.~Theorem~\ref{monomial}),  the  geodesic g-convexity is equivalent to the geodesic g-convexity.  For additively separable functions,  we are not able to prove this equivalence.

\item[$\diamond$] Theorem~\ref{snc_for_seperable_case} holds for any smooth additively separable function.  Moreover,  condition~\eqref{snc_for_seperable_case:item1} is clearly equivalent to that  for some $1 \le t \le s$,  $f_t$ has no critical point.
\end{itemize}
\end{remark}

\begin{example}\label{projection of the higher dimension}
We consider $f = x_1^3+x_2 \in \mathbb{R}[x_1,x_2]$.  Then $\grad_{\mathsf{e}} f = (3x_1^2,1)^\tp$ which is nonzero for all $(x_1,x_2) \in\mathbb{R}^2$.  Thus by Proposition~\ref{thm:sufficient0},  $f$ is  g-convex with respect to some connection.  On the other side,  $f$ is additively separable since $f(x_1,x_2)=f_1(x_1)+f_2(x_2)$ where $f_{1}(x_1)=x_1^3,f_{2}(x_{2})=x_2$.  According to Theorem~\ref{snc_for_dim1},  $f_1$ is not  g-convex with respect to any  connection.
\end{example}
As an application,  we apply Theorem~\ref{snc_for_seperable_case} to additively separable polynomials.  Given positive integers $n$ and $d$ and a positive real number $r$,  we denote 
\[
S_{n,d}(r) \coloneqq \left\lbrace \sum_{j=1}^n f_j: f_j = \sum_{k=0}^d a_{jk} x_j^k \in \mathbb{R}[x_j]_{\le d},  \;  |a_{jk}| \le r,  \;  1 \le j \le n, \;  1 \le k \le d
\right\rbrace.
\]
We recall from \eqref{eq:Pd} that 
\[
\left\lbrace \sum_{k=0}^d a_{jk} x_j^k \in \mathbb{R}[x_j]_{\le d},  \;  |a_{jk}| \le r,  \; 1 \le k \le d
\right\rbrace \simeq [-r,r]^{d+1}
\] for each $1 \le j \le n$.  This leads to $S_{n,d}(r) \simeq [-r,r]^{dn} \times [-nr,nr]$.  Moreover,  $S_{n,d}(\infty) \coloneqq \bigcup_{r > 0} S_{n,d}(r)$ consists of additively separable polynomials of degree at most $d$.
\begin{theorem}[Dimension and density IV]\label{snc_for_seperable_case_univariate}
For any positive integers $n$ and $d$,  $\dim (S_{n,d}(\infty) \cap A_{n,d}) = nd+1$.  Moreover,  for each $r > 0$,  if we equip $S_{n,d}(r)$ with the uniform probability distribution,  then the probability that $f \in S_{n,d}(r)$ is  g-convex with respect to some  connection is $O(1/\log d)$.
\end{theorem}
\begin{proof}
This is a direct consequence of Theorems~\ref{snc_for_seperable_case} and \ref{prob_dim_1}.
\end{proof}

\bibliographystyle{plain}
\bibliography{ref}

\end{document}